\documentclass[a4paper]{amsart}

\usepackage[utf8]{inputenc}			%for utf8 characters in pdfLatex
\usepackage[T1]{fontenc}			%for accents encoding
\usepackage[british]{babel} 		%british english

\usepackage{
amsmath,			%\DeclareMathOperator, etc
amsthm,				%\newtheorem, etc
amssymb,			%amsfonts, \mathbb, \cap, \cup, etc
latexsym,			%\unlhd, \sqsubset, \leadsto, \Join, etc
xfrac,				%lateral fractions \sfrac
mathtools,			%\coloneqq, \mathclap, \prescript, etc
bbm,				%several blackboard fonts
enumitem,			%enumeration options (alternative enumerate)
caption,			%for caption management \caption, \captionsetup
xcolor,				%for \color and \definecolor
tikz,				%for drawing graphs and diagrams
microtype,			%microtyping (eliminates most badbox errors)
hyphenat,			%for \hyp, allows hyphenating words with -
hyperref,			%hyperlinks 
cleveref,			%for \cref
adjustbox			%for adjusting size
}
\usepackage[mathscr]{eucal}

\newtheorem{theo}{Theorem}[section]
\newtheorem{lem}[theo]{Lemma}
\newtheorem{coro}[theo]{Corollary}
\newtheorem{que}[theo]{Question}

\theoremstyle{definition}
\newtheorem{defi}[theo]{Definition}%[section]
\newtheorem{ex}[theo]{Example}
\newtheorem{rmk}[theo]{Remark}
\newtheorem*{nota}{Notation}

\crefname{theo}{Theorem}{Theorems}
\crefname{lem}{Lemma}{Lemmas}
\crefname{coro}{Corollary}{Corollaries}
\crefname{que}{Question}{Questions}
\crefname{defi}{Definition}{Definitions}
\crefname{ex}{Example}{Examples}
\crefname{rmk}{Remark}{Remarks}
\crefname{enumi}{}{}
\crefname{section}{Section}{Sections}

\renewcommand*\qedsymbol{\setbox0=\hbox{\quad\footnotesize{\normalfont Q.E.D.}}\kern\wd0 \strut \hfill \kern-\wd0 \box0}

\newcommand*\N{\mathbb{N}}
\newcommand*\Z{\mathbb{Z}}

\newcommand*\sth{:}

\newcommand*\map{\colon}

\catcode`@=11 \def\rfrac#1#2{\def\rf@num{#1}\def\rf@den{#2}\mathpalette\rfr@c\relax} \def\rfr@c#1#2{\raise.5ex\hbox{$\mathsurround=0pt #1\rf@num$}\kern-0.1em/\kern-0.1em\lower.25ex\hbox{$\mathsurround=0pt #1\rf@den$}}\catcode`@=12

\newcommand*\inc{{\iota}}
\newcommand*\id{{\mathrm{id}}}

\newcommand*\For{{\mathrm{For}}}
\newcommand*\LFor{{\mathrm{LFor}}}
\newcommand*\Diag{{\mathrm{Diag}}}
\newcommand*\premodels{\mathrel{\mathord{\models}_\star}}
\newcommand*\premodelspc{\mathrel{\mathord{\models}^{\mathrm{pc}}_\star}}
\newcommand*\dd{{\mathrm{d}}}
\newcommand*\lang{{\mathtt{L}}}
\newcommand*\theory{{\mathtt{T}}}
\newcommand*\Th{{\mathrm{Th}}}

\newcommand*\loc{{\mathrm{loc}}}
\newcommand*\pc{{\mathrm{pc}}}
\newcommand*\at{{\mathrm{at}}}

\newcommand*\Dtt{\mathtt{D}}
\newcommand*\Sorts{\mathtt{S}}
\newcommand*\BoundConst{\mathtt{B}}
\newcommand*\Bound{\mathrm{B}}
\newcommand*\toprel{\mathrel{\top}}
\newcommand*\Mod{\mathcal{M}}

\title{Completeness in local positive logic}

\author{Arturo Rodr\'{\i}guez Fanlo} 
\thanks{Rodr\'{\i}guez Fanlo has carried out this work at the Hebrew University of Jerusalem while supported by the Israel Academy of Sciences and Humanities \& Council for Higher Education Excellence Fellowship Program for International Postdoctoral Researchers and partially supported by STRANO PID2021-122752NB-I0.}

\address{Autonomous University of Madrid, Ciudad Universitaria de Cantoblanco, 28049, Madrid, España.}
\email[A.~Rodr\'{\i}guez Fanlo]{arturo.rodriguez@uam.es}

\author{Ori Segel}

\address{Einstein Institute of Mathematics, Hebrew University of Jerusalem, 91904, Jerusalem, Israel.}
\email[O.~Segel]{Ori.Segel@mail.huji.ac.il}


\begin{document}
\begin{abstract}
We develop the basic model theory of local positive logic, a new logic that mixes positive logic (where negation is not allowed) and local logic (where models omit types of infinite distant pairs). We study several basic model theoretic notions such as compactness, positive closedness (existential closedness) and completeness (irreducibility).\end{abstract}

\maketitle

\section*{Introduction}
It is common practice in mathematics to modify basic objects with more and more structure for the sake of generality. In model theory, well{\hyp}known examples are continuous logic \cite{yaacov2008model} and positive logic \cite{yaacov2007fondements,poizat2008positive,yaacov2003positive} as generalisations of first{\hyp}order logic. In this paper we present the basic framework of a new logic, specially underlining the differences and similarities between this new logic and usual positive logic.

Local positive logic was first introduced by Hrushovski in \cite{hrushovski2022lascar} and applied in his study of approximate subgroups. Local positive logic is the combination of positive logic and local logic. While positive logic is well{\hyp}known, local logic was also first introduced in \cite{hrushovski2022lascar}. 

Local logic is a variation of first{\hyp}order logic that aims to relax compactness. Compactness is a fundamental tool in model theory, but its full power is sometimes a hindrance. Local logic attempts to replace the role of compactness in first{\hyp}order logic with ``local compactness''. It does this by restricting quantifiers to \emph{local quantifiers}, i.e. each quantifier imposes that the variable is located in some ``ball''. Although local logic has several features that distinguish it from first{\hyp}order logic, at the same time this extension is very lightweight in the sense that, by adding constants, local logic can be mostly reduced to sorted first{\hyp}order logic by adding a new sort for each ball centred at the constants. 

Positive logic was born to generalise Robinson's work on existential closedness. Existential closedness is a fundamental notion in model theory capturing the idea of a ``generic'' or ``limit'' model. In positive logic this notion is modified by allowing arbitrary homomorphisms in place of just embeddings, obtaining what we here call \emph{positive closedness}. Formally, this is achieved by restricting the first{\hyp}order syntax exclusively to \emph{positive formulas} (i.e. formulas without negations). As a consequence, positive logic also eliminates the total disconnectedness nature of first{\hyp}order logic. 

As already observed in \cite{yaacov2003positive}, there is no substantial difference between definable and type{\hyp}definable sets in positive logic, which makes it a suitable syntax for hyperdefinable sets (i.e. quotients by type{\hyp}definable equivalence relations). In turn, the original motivation for the development of local logic in \cite{hrushovski2022lascar} was to introduce a natural syntax for the study of locally definable sets, more specifically of definably generated groups. It is therefore natural for us to study the combination of both logics, our original motivation being the development of a syntax for locally hyperdefinable sets, more precisely hyperdefinably generated groups.

We now briefly explain our intended application. We aim to study a group $G$ generated by a symmetric hyperdefinable set $X$ in some first-order structure. We can study $G$ as a local structure in positive logic where the locality relations $\{\dd_n\}_{n\in\N}$ are given by $\dd_n(x,y)\Leftrightarrow x^{-1}y\in X^n$ and we have a relation symbol for every hyperdefinable subset. Alternatively, we can instead consider $G$ as a torsor, rather than a group: in this case, for each $n${\hyp}ary hyperdefinable subset $V$ we add a $2n${\hyp}ary relation symbol with intended interpretation $R_V(\bar{x},\bar{y})\Leftrightarrow (x^{-1}_iy_i)_{i<n}\in V$ --- considering $G$ as a torsor has the advantage that $G$ becomes a subgroup of the group of automorphisms of the structure. The use of local logic ensures that the structure as group is generated by $X$. On the other hand, the use of positive logic captures the hyperdefinable nature of $G$.

To further motivate the introduction of this new logic, we present now a typical case where we think that the introduction of local positive logic might be of interest. This is the case of the study of connected graphs with extra structure. As positive logic was developed to generalise Robinson's work on existentially closed models, the study of graphs with extra structure is a typical example where positive logic is applied in order to obtain universal models analogous to the random graph. Now, these analogues are not always suitable if one is focused on connected graphs. 

For instance, let us take $\Z$ as a graph (i.e. with a binary relation for graph distance $\leq 1$) with unary predicates for the intervals; the associated universal positively closed model turns out to be $\Z\cup\{\pm\infty\}$ which can easily be understood as the ``limit'' obtained from the substructures $(-\infty,-n^2]\cup [-n,n] \cup [n^2,\infty)$ when $n\rightarrow \infty$. We can immediately see that this structure has three critically different parts (namely, $\Z$, $\infty$ and $-\infty$). More precisely, $\Z$ as graph can be visualised as a line and, while $\Z$ is connected (as graph), the resulting universal positively closed model $\Z\cup\{\pm\infty\}$ is disconnected with three different connected components. This is a natural consequence of the fact that positive logic takes into account all substructures of $\Z$, regardless of their connectedness. Local positive logic restricts our attention to just ``connected'' structures, allowing us to obtain only the ``connected limits''. In this example, instead of $\Z\cup\{\pm\infty\}$, we separately obtain the three different connected components as locally positively closed models.\medskip

This paper should be considered a continuation on \cite[Section 2]{hrushovski2022lascar}. For the remaining of the introduction, we outline the structure of the paper and remark the main new results with respect to Hrushovski{'s} work.

In \cref{s:local languages and structures}, we define \emph{local languages} and \emph{local structures} in the many sorted case and with arbitrary local signatures (\cref{d:local language,d:local structure}), and prove a compactness theorem for this logic (\cref{t:local compactness}). In \cite{hrushovski2022lascar}, local positive logic is defined only in the one sort case. While in standard positive logic the generalisation to many sorts is mostly straightforward, in local positive logic a few technical subtleties appear. For instance, the presence of several sorts brings to the surface the notion of \emph{pointed variables} (\cref{d:pointed}), which plays a significant role in several places of the paper. 

In \cref{s:local positive closedness}, we recall the definition of \emph{locally positive closed models} (called existential closed models in \cite{hrushovski2022lascar}) in \cref{d:locally positively closed}, proving their existence in \cref{t:positively closed} --- this result was already stated (with omitted proof) in \cite[Section 2]{hrushovski2022lascar}. We also recall the definition of \emph{denials of formulas} (\cref{d:denials}) and prove their existence in \cref{l:denials}, which was stated without proof in \cite[Section 2]{hrushovski2022lascar}. We also show that positive substructures of locally positively closed models are locally positively closed models in \cref{c:positive substructures}, a fact that was implicitly assumed in \cite{hrushovski2022lascar}.

The study of completeness in local positive logic is the main topic of this paper. We define several completeness properties (\cref{d:irreducibility}) and discuss their relations in \cref{s:completeness}.  In positive logic, completeness is already a more complicated topic than in standard first order logic --- with some properties that are equivalent in standard first order logic no longer being equivalent. As we show here, the situation in local positive logic is even more complicated (\cref{t:weakly complete iff irreducible,t:local irreducibility,e:local irreducibility}) as many properties which are still equivalent in positive logic are not longer so in local positive logic.
% While in standard positive logic all the completeness properties are equivalent, this is no longer true for local positive logic (\cref{t:weakly complete iff irreducible,t:local irreducibility,e:local irreducibility}). 
We conclude the section studying weak completions and completions (\cref{l:weak completions,l:existence of weak completions,l:maximal and completions,e:maximal not complete}). We conclude the paper discussing when local positive logic reduces to standard positive logic in \cref{s:inherent locality}, giving various equivalent formulations in \cref{t:inherently local}.

This is the first of a series of papers which aims to extend several results of \cite{hrushovski2022lascar}. In subsequent papers we will study types and saturation in local positive logic. We will also study definability patterns for local positive logic, extending \cite{hrushovski2022definability,segel2022positive}. Finally, we will apply the definability patterns to study hyperdefinable approximate subgroups and rough approximate subgroups. Beyond some specific questions, we leave as an open task to adapt more elaborated model theoretic machinery (e.g. classification theory) to local positive logic. 

\section{Local languages and structures} \label{s:local languages and structures}
\begin{defi} \label{d:local language}
A \emph{local (first{\hyp}order) language} is a first{\hyp}order language\footnote{As it is common practice in positive logic, we consider that every first{\hyp}order language contains a binary relation of equality $=$ for each sort and two $0${\hyp}ary relations $\bot$ and $\top$, where $\top$ is always true and $\bot$ is always false.} $\lang$ with the following additional data:
\begin{enumerate}[label={\rm{(\roman*)}}, wide]
\item It is a relational language with constants, i.e. it has no function symbols except for constant symbols.
\item It has a distinguished partially ordered commutative monoid of binary relation symbols $\Dtt^s\coloneqq (\Dtt^s,\dd^0,\ast_{s},\prec_s)$ for each single sort $s$. The binary relation symbols in $\Dtt^s$ are the \emph{locality relation symbols} on the sort $s$. The identity $\dd^0$ of $\Dtt^s$ is the equality symbol on sort $s$ and it is also the least element for $\prec_s$. 
\item There is a function $\BoundConst$ assigning to any pair $c_1\, c_2$ of constant symbols of the same sort a locality relation symbol $\BoundConst_{c_1,c_2}$ on that sort. We call $\BoundConst$ the \emph{bound function} and $\BoundConst_{c_1,c_2}$ the \emph{bound} for $c_1$ and $c_2$.
\end{enumerate}
\end{defi}
We call $(\Dtt,\BoundConst)$ the \emph{local signature} of $\lang$. A \emph{sort} is a tuple of single sorts; its \emph{arity} is its length. We write $\Sorts$ for the set of single sorts. A \emph{variable} $x$ is a set of single variables; its \emph{arity} is its size. The \emph{sort} of a variable $x$ is the sort $s=(s_{x_i})_{x_i\in x}$ where $s_{x_i}$ is the sort of $x_i$.

\begin{rmk} Let us make some remarks about \cref{d:local language}, highlighting possible generalisations:
\begin{enumerate}[label={\rm{(\roman*)}}, wide] 
\item We only study relational languages with constants. This significantly simplifies the definition of the bound function $\BoundConst$. In many situations, one could deal with function symbols by using relations defining their graphs. This is common practice in positive logic, since those relations are forced to define graphs of functions in positively closed models. We leave as a task to develop local positive logic for languages with function symbols. The key issue is to find a good definition of $\BoundConst$ for function symbols that is not too restrictive.
\item We only consider locality relations between elements of the same sort. For locality relations between different sorts, it might be more natural to merge all these sorts in a new one and replace them by disjoint unary predicates. For positively closed models, these predicates would give a partition of the new sort.
\item Locality relations are binary. We leave as an open task the development of a higher{\hyp}arity local logic.    
\end{enumerate}
\end{rmk}

A \emph{local reduct} is a local language which is a reduct; a \emph{local expansion} is a local language which is an expansion. As usual, we write $\lang(C)$ for the expansion of $\lang$ given by adding the symbols in $C$. Note that, when adding constant symbols, the expansion should also extend the bound function $\BoundConst$. Among all the local expansions, we denote by $\lang_\star$ the expansion of $\lang$ that consists on ``forgetting'' the local signature, i.e. $\lang_\star$ is the expansion of $\lang$ given by adding to each sort a maximal locality relation symbol $\dd_{\star}$ whose intended interpretation is trivial (i.e. $\dd_\star(a,b)$ for every pair $a,b$ in the relevant sort).  

\begin{defi}\label{d:pointed} We say that a local language $\lang$ is \emph{pointed} if it has at least one constant symbol for every sort. We say that a set of parameters (new constants) $A$ is \emph{pointed} for $\lang$ if the expansion $\lang(A)$ is pointed. We say that a variable $x$ of sort $s$ is \emph{pointed} in $\lang$ if any evaluation of it is pointed. In other words, a variable is pointed if there is at least one constant symbol for every sort that does not appear in $s$. We say that a sort $s$ is \emph{pointed} if any variable on it is pointed, i.e. if there is at least one constant symbol for every sort that does not appear in $s$. 
\end{defi}

For locality relation symbols $\dd_1,\ldots,\dd_n$ of respective single sorts $s_1,\ldots,s_n$, we write $\dd_1\times\cdots\times \dd_n(x,y)\coloneqq \bigwedge_{1\leq i\leq n} \dd_i(x_i,y_i)$ for the \emph{product} of the locality relation symbols. A \emph{locality relation symbol} of sort $s=(s_1,\ldots,s_n)$ is a product of locality relation symbols of the respective single sorts. We write $\Dtt^s$ for the locality relation symbols of sort $s$ and pointwise extend $\ast$ and $\prec$ to products.\medskip

Let $\For(\lang)$ denote the set of first{\hyp}order formulas of $\lang$. Given a variable $x$, let $\For^x(\lang)$ denote the subset of formulas whose free single variables are contained in $x$ --- in particular, $\For^0(\lang)$ is the set of sentences of $\lang$. 

A formula is \emph{positive} if it is built from atomic formulas by conjunctions, disjunctions and existential quantifiers; let $\For_+$ denote the set of positive formulas. A formula is \emph{primitive positive} if it is  built from atomic formulas by conjunctions and existential quantifiers. A formula is \emph{negative} if it is the negation of a positive formula. A formula is \emph{primitive negative} if it is the negation of a primitive positive formula. 

Let $\varphi$ be a formula, $x$ a single variable, $\dd$ a locality relation on the sort of $x$ and $t$ a term of the same sort such that $x$ does not appear in $t$ (i.e. a constant or a variable different to $x$). The \emph{local existential quantification of $\varphi$ on $x$ in $\dd(t)$} is the formula $\exists x\in \dd(t)\mathrel{} \varphi\coloneqq \exists x\mathrel{}(\dd(x,t)\wedge \varphi)$. The \emph{local universal quantification of $\varphi$ on $x$ in $\dd(t)$} is the formula $\forall x\in\dd(t)\mathrel{}\varphi\coloneqq \neg\exists x\in\dd(t)\mathrel{}\neg\varphi=\forall x\mathrel{}(\dd(x,t)\rightarrow \varphi)$.

A formula is \emph{local} if it is built from atomic formulas using negations, conjunctions and local existential quantification; let $\LFor$ denote the set of local formulas. 
A formula is \emph{local positive} if it is built from atomic formulas using conjunctions, disjunctions and local existential quantification; let $\LFor_+$ denote the set of local positive formulas. A formula is \emph{local primitive positive} if it is  built from atomic formulas by conjunctions and local existential quantifiers. A formula is \emph{local negative} if it is the negation of a local positive formula.

Finally, we define $\Pi_1${\hyp}local formulas. A formula is \emph{$\Pi_1${\hyp}local} if it is of the form $\forall x\mathrel{}\varphi$ where $\varphi$ is a local formula. %A formula is \emph{$\Sigma_0${\hyp}local} if it is a local formula. A formula is \emph{$\Pi_n${\hyp}local} if it is the negation of a $\Sigma_n${\hyp}local formula. A formula is \emph{$\Sigma_{n+1}${\hyp}local} if it is of the form $\exists x\mathrel{} \varphi$ where $\varphi$ is a $\Pi_n${\hyp}local formula. We write $\Pi_n$ for the set of $\Pi_n${\hyp}local formulas and $\Sigma_n$ for the set of $\Sigma_n${\hyp}local formulas.

\begin{rmk} \label{r:normal forms} We have the following normal forms:
\begin{enumerate}[label={•}, wide]
\item Every primitive positive formula is equivalent to $\exists x\mathrel{} \theta$ for some finite variable $x$ and some quantifier free primitive positive formula $\theta$.
\item Every positive formula is equivalent to a finite disjunction of primitive positive formulas. Every positive formula is equivalent to $\exists x\mathrel{} \theta$ for some finite variable $x$ and some quantifier free positive formula $\theta$. In particular, every positive formula is existential and every negative formula is universal.
\item Every local formula is equivalent to $Q_1x_1\in\dd_1(t_1)\ldots Q_nx_n\in\dd_n(t_n)\mathrel{} \theta$ for some quantifiers $Q_1,\ldots,Q_n$, some single variables $x_1,\ldots,x_n$, some terms $t_1,\ldots,t_n$, some locality relation symbols $\dd_1,\ldots,\dd_n$ and some quantifier free formula $\theta$. 
\item Every local positive formula is equivalent to $\exists x_1\in\dd_1(t_1)\ldots \exists x_n\in\dd_n(t_n)\mathrel{} \theta$ for some single variables $x_1,\ldots,x_n$, some terms $t_1,\ldots,t_n$, some locality relation symbols $\dd_1,\ldots,\dd_n$ and some quantifier free positive formula $\theta$.
\end{enumerate}
\end{rmk}

%\begin{rmk} Note that, up to elementary equivalence, $\Pi_n$ is closed under local existential quantifiers for any $n\in\N$. Indeed, it is tautological for $\Sigma_0=\Pi_0$. Now, suppose it holds for $\Pi_{n-1}$ and consider $\varphi\coloneqq \forall y\mathrel{} \theta$ with $\theta\in \Sigma_{n-1}$. Then, $\exists x\in\dd(t)\mathrel{} \varphi$ is elementary equivalent to $\forall y' \exists x\in \dd(t) \forall y\in \dd^0(y')\mathrel{} \theta$ where $y'$ is an auxiliary variable on the same sort as $y$ and $\underline{\dd}^0$ is the identity locality relation on that sort. It follows by induction hypothesis that $\exists x\in\dd(t) \forall y \in \dd^0(y')\mathrel{} \theta$ is elementary equivalent to a formula in $\Sigma_{n-1}$. Hence, $\varphi$ is elementary equivalent to a formula in $\Pi_n$. \end{rmk}
From now on, unless otherwise stated, fix a local language $\lang$.
\begin{defi} \label{d:local structure} A \emph{local $\lang${\hyp}structure} is an $\lang${\hyp}structure\footnote{We only consider structures such that no sort is empty.} $M$  satisfying the following locality axioms:
\begin{enumerate}[label={\rm{(A{\arabic*})}}, ref={\rm{A{\arabic*}}}, wide]
\item \label{itm:axiom 1} Every locality relation symbol is interpreted as a symmetric binary relation.
\item \label{itm:axiom 2} Any two locality relation symbols $\dd_1$ and $\dd_2$ on the same sort with $\dd_1\preceq \dd_2$ satisfy $\dd^M_1\subseteq \dd^M_2$.
\item \label{itm:axiom 3} Any two locality relation symbols $\dd_1$ and $\dd_2$ on the same sort satisfy $\dd^M_1\circ \dd^M_2\subseteq (\dd_1\ast \dd_2)^M$.
\item \label{itm:axiom 4} Any two constant symbols $c_1$ and $c_2$ of the same sort satisfy $\BoundConst_{c_1,c_2}(c_1,c_2)$.
\item \label{itm:axiom 5} Any two single elements $a$ and $b$ of the same sort satisfy $\dd(a,b)$ for some locality relation $\dd$ on that sort. 
\end{enumerate}

For any element $a$ and locality relation $\dd$ on its sort, the \emph{$\dd${\hyp}ball at $a$} is the set $\dd(a)\coloneqq \{b\sth M\models \dd(a,b)\}$.
\end{defi}

Expansions and reducts of local structures are defined in the usual way. 

%\begin{nota} We often use underline notation to distinguish between syntax and semantics. Given a definable set $X$, we write $\underline{X}$ for the formula defining $X$. This notation is particularly useful to describe expansions: in that case, given a semantic object $X$, we write $\underline{X}$ for the predicate that we add to the language whose intended interpretation is $X$. We always follow this underline notation except in the case of constant symbols and parameters, when we often omit the underline if there is no confusion.\end{nota}

\begin{rmk} The first four locality axioms could be easily expressed by first{\hyp}order universal axiom schemes. We write $\theory_\loc(\lang)$ for the universal first{\hyp}order theory given by axioms \cref{itm:axiom 1,itm:axiom 2,itm:axiom 3,itm:axiom 4} --- we omit $\lang$ if it is clear from the context. On the other hand, \cref{itm:axiom 5} says that some particular quantifier free binary types are omitted, and it could be expressed as a \emph{geometric axiom} (i.e. a positive formula allowing infinite disjunctions too, see \cite[\S D.1]{johnstone2002sketches}).
\end{rmk}

Let $\lang$ be a local language and $\Gamma$ a set of sentences of $\lang$. An \emph{$\lang${\hyp}local model} of $\Gamma$ is a local $\lang${\hyp}structure satisfying it. We typically omit $\lang$ if it is clear from the context. We write $\Mod(\Gamma/\lang)$ for the class of $\lang${\hyp}local models of $\Gamma$. We say that $\Gamma$ is \emph{$\lang${\hyp}locally satisfiable} if it has an $\lang${\hyp}local model. We say that $\Gamma$ is \emph{finitely $\lang${\hyp}locally satisfiable} if every finite subset of $\Gamma$ is $\lang${\hyp}locally satisfiable.

%\begin{rmk} In classic first{\hyp}order logic, the set of formulas $\For(\lang)$ implicitly specifies the language $\lang$. Though trivial, this remark is fundamental: it means that satisfaction of a set of sentences is independent of (expansions of) the language. In local logic, however, the set $\For(\lang)$ does not specify the local signature. Consequently, as being a local model depends on the local signature, it may change by taking local expansions of the language. %Therefore, rigorously speaking, we should always say $\lang${\hyp}local model. The same technical issue shows up in many places through the paper; all the definitions that involve the word ``local'' typically depend on the particular choice of the local signature of the language and, therefore, are not well behaved under local reducts/expansions. Nonetheless, aiming to simplify the terminology, we often omit $\lang$ unless some clarification is needed. \end{rmk}

Given two subsets of sentences $\Gamma_1$ and $\Gamma_2$, we say that $\Gamma_1$ \emph{$\lang${\hyp}locally implies} $\Gamma_2$, written $\Gamma_1\models_\lang \Gamma_2$, if every $\lang${\hyp}local model of $\Gamma_1$ is an $\lang${\hyp}local model of $\Gamma_2$. We say that $\Gamma_1$ and $\Gamma_2$ are \emph{$\lang${\hyp}locally elementary equivalent} if they have the same local models. %We typically omit $\lang$ if it is clear from the context.
\begin{nota} We write $M\models_\lang \Gamma$ to say that $M$ is an $\lang${\hyp}local model of $\Gamma$. We write $\premodels$ (i.e. $\models_{\lang_\ast}$) for satisfaction for non{\hyp}local structures. Similarly, we write $\Mod_\star(\Gamma)$ for the class of models of $\Gamma$ regardless of localness.
\end{nota}

A \emph{negative $\lang${\hyp}local theory} $\theory$ is a set of negative sentences closed under local implication. The \emph{associated positive counterpart} $\theory_+$ is the set of positive sentences whose negations are not in $\theory$. In other words, a positive formula $\varphi$ is in $\theory_+$ if and only if there is a local model $M\models_\lang \theory$ satisfying $\varphi$; we say in that case that $M$ is a \emph{witness of $\varphi\in\theory_+$}. We write $\theory_{\pm}\coloneqq \theory\wedge\theory_+$\footnote{We write $\Gamma\wedge\Gamma'\coloneqq\Gamma\cup \Gamma'$ for sets of formulas.}. %We typically omit $\lang$ if it is clear from the context.
 
%More generally, one can talk about existential $\lang${\hyp}local theories, universal $\lang${\hyp}local theories, $\Sigma_n${\hyp}local ($\lang${\hyp}local) theories, and so forth. %A \emph{""" $\lang${\hyp}local theory} is a subset of """ sentences closed under $\lang${\hyp}local implication. However, in this paper we only discuss negative local theories. 

We say that $\theory$ is \emph{$\lang${\hyp}locally axiomatisable} by a subset $\Gamma\subseteq \theory$ if $\Gamma$ is $\lang${\hyp}locally elementary equivalent to $\theory$. Abusing notation, we may use $\Gamma$ to denote $\theory$ or \textsl{vice versa}. 

\begin{rmk} The associated \emph{primitive negative local theory} is the subset of primitive negative sentences of $\theory$. %The \emph{associated primitive positive counterpart} is the subset of primitive positive sentences of $\theory_+$. 
Every sentence in $\theory$ is (elementary equivalent to) a finite conjunction of primitive negative sentences, so the associated primitive negative local theory always (locally) axiomatises $\theory$. In sum, we could generally work only with primitive negative sentences; this is very useful in practical examples and was repeatedly used in \cite{hrushovski2022definability,segel2022positive}. 
\end{rmk}
 
The first fundamental result in local logic is the following easy lemma reducing local satisfiability to usual satisfiability.
\begin{lem} \label{l:locally satisfiable} Let $\Gamma$ be a set of $\Pi_1${\hyp}local sentences. Then, $\Gamma$ is locally satisfiable if and only if $\Gamma\wedge \theory_\loc$ is satisfiable. Furthermore, suppose $N\premodels \Gamma\wedge \theory_\loc$ and pick $o\coloneqq (o_s)_{s\in\Sorts}$ with $o_s\in N^s$ for each single sort $s$. If there are constants of sort $s$, assume that $\dd(o_s,c^N)$ for some constant $c$ of sort $s$ and some locality relation $\dd$. Then, the \emph{local component $\Dtt(o)$ at $o$}, given by $\Dtt(o)^s\coloneqq\bigcup_{\dd\in\Dtt^s}\dd(o_s)$ for each sort $s$, is a local model of $\Gamma\wedge\theory_\loc$.
\end{lem}
\begin{proof} Obviously, any local model of $\Gamma$ is a model of $\Gamma\wedge\theory_\loc$. On the other hand, suppose that $N\premodels\Gamma\wedge\theory_\loc$, $(o_s)_{s\in\Sorts}$ and $\Dtt(o)$ as in the statement. By \cref{itm:axiom 4}, every constant symbol is interpreted in $\Dtt(o)$, so it is the universe of a substructure of $N$. As $\theory_\loc$ is universal, it follows that $\Dtt(o)$ satisfies \cref{itm:axiom 1,itm:axiom 2,itm:axiom 3,itm:axiom 4}. Also, \cref{itm:axiom 5} trivially holds: for any $a,b\in \Dtt(o)^s$ there are locality relations $\dd_1$ and $\dd_2$ such that $\dd_1(a,o_s)$ and $\dd_2(b,o_s)$, so we have $\dd_1\ast \dd_2(a,b)$.

By a straightforward induction on the complexity, we have that, for any $\varphi(x)\in \LFor^x(\lang)$ and $a\in \Dtt(o)^x$, $\Dtt(o)\models_\lang\varphi(a)$ if and only if $N\premodels\varphi(a)$.
%\begin{enumerate}[label={•}, wide]
%\item For $\varphi$ atomic there is nothing to prove.
%\item Say $\varphi=\varphi_1\wedge\varphi_2$. By induction hypothesis, $N\premodels\varphi_1(a)$ and $N\premodels\varphi_2(a)$ if and only if $M\models_\lang\varphi_1(a)$ and $M\models_\lang\varphi_2(a)$. Hence, $N\premodels\varphi(a)$ if and only if $M\models_\lang\varphi(a)$.
%\item Say $\varphi=\exists y\in \dd(t(x))\mathrel{} \theta(x,y)$. Suppose $N\premodels\varphi(a)$. Then, there is $b\in \dd(t^{N}(a))$ with $N\premodels\theta(a,b)$. Note that $t^{N}(a)=t^M(a)$, so $b$ is in $M$. Thus, by induction hypothesis, $M\models_\lang\theta(a,b)$ and $M\models_\lang \dd(t(a),b)$, concluding $M\models_\lang\varphi(a)$.
%
%Conversely, suppose $M\models_\lang\varphi(a)$. Then, there is $b\in \dd(t^M(a))$ with $M\models_\lang \theta(a,b)$. By induction hypothesis, we conclude that $N\premodels \theta(a,b)$ with $b\in \dd(t^{N}(a))$, so $N\premodels\varphi(a)$.   
%\item Say $\varphi=\neg\theta$. By induction hypothesis, $N\premodels\theta(a)$ if and only if $M\models_\lang\theta(a)$. Hence, $N\premodels\varphi(a)$ if and only if $N\models_\lang\varphi(a)$.
%\end{enumerate} 
Hence, take $\varphi\in \Gamma$ and suppose $\varphi=\forall x\mathrel{}\theta$ with $\theta$ a local formula. As $N\premodels\varphi$, in particular, $N\premodels \theta(a)$ for any $a\in \Dtt(o)^x$, concluding that $\Dtt(o)\models_\lang \theta(a)$  for any $a\in \Dtt(o)^x$. So, $\Dtt(o)\models_\lang\varphi$ and, as $\varphi$ was arbitrary, $\Dtt(o)\models_\lang\Gamma$.
\end{proof}

As an immediate corollary, we get the following compactness theorem for local logic.
\begin{theo}[Compactness theorem] \label{t:local compactness} Let $\Gamma$ be a set of $\Pi_1${\hyp}local sentences. Then, $\Gamma$ is locally satisfiable if and only if it is finitely locally satisfiable.
\end{theo}

Every substructure of a local structure is a local structure too, as $\theory_\loc$ is universal. Therefore, we can freely apply the downwards L\"{o}wenheim{\hyp}Skolem Theorem in local logic. For the upward L\"owenheim{\hyp}Skolem Theorem we have the following straightforward version.

\begin{theo}[L\"{o}wenheim{\hyp}Skolem Theorem] Let $\Gamma$ be a set of $\Pi_1${\hyp}local sentences and $\kappa\geq |\lang|$ a cardinal. Let $S$ be a subset of single sorts of $\lang$. Suppose that $\Gamma$ has a local model with some infinite ball on $s$ for each $s\in S$. Then, $\Gamma$ has a local model $M$ such that $|M^s|=\kappa$ for each $s\in S$.
\end{theo}
\begin{proof} Let $N_0$ be a local model of $\Gamma$ with an infinite ball on $s$ for each $s\in S$. Take a (non{\hyp}local) $\kappa${\hyp}saturated elementary extension $N$ of $N_0$. By hypothesis and saturation of $N$, there is $o=(o_s)_{s\in\Sorts}$ with $o_s\in N^s$ such that, if there is a constant $c$ of sort $s$, $\dd(o_s,c^N)$ for some locality relation $\dd$, and, if $s\in S$, $|\dd(o_s)|\geq \kappa$ for some locality relation $\dd$. Thus, by \cref{l:locally satisfiable}, the local component $\Dtt(o)$ of $N$ at $o$ realises $\Gamma$. Obviously, $|\Dtt(o)^s|\geq \kappa$ for each $s\in S$. Applying downwards L\"owenheim{\hyp}Skolem Theorem now, we can find $M\prec \Dtt(o)$ such that $|M^s|=\kappa$ for $s\in S$.
\end{proof}

For a local structure $M$, we write $\Th_-(M)$ for its negative theory, i.e. the set of negative sentences satisfied by $M$. For a subset $A$, we write $\Th_-(M/A)$ for the theory of the expansion of $M$ given by adding parameters for $A$. Note that $\Th_-(M)$ must be a negative local theory when $M$ is local. We say that a negative local theory $\theory$ is \emph{($\lang${\hyp}locally) weakly complete} if there is a local structure $M$ such that $\theory=\Th_-(M)$ (see \cref{d:irreducibility} for more details).

\begin{rmk} \label{r:locally weakly complete} Let $\theory$ be a negative local theory and $M$ a local structure. Note that $\theory=\Th_-(M)$ if and only if $M\models_\lang \theory_{\pm}$. In particular, a negative local theory $\theory$ is weakly complete if and only if $\theory_{\pm}$ is locally satisfiable.
%\begin{proof} Obviously, $M\models_\lang\theory$ if and only if $\theory\subseteq \Th_-(M)$. On the one hand, suppose $M\models_\lang\theory_{\pm}$, so $\theory\subseteq \Th_-(M)$. Then, for any negative sentence $\varphi\notin\theory$, we have $\neg\varphi\in \theory_+$, so $M\models_\lang\neg\varphi$. In other words, we have $M\not\models_\lang\varphi$, so $\varphi\notin\Th_+(M)$. As $\varphi\notin\theory$ is generic, we conclude $\theory=\Th_-(M)$. On the other hand, suppose $\theory=\Th_-(M)$, so $M\models_\lang\theory$. Then, for any positive sentence $\varphi\in\theory_+$, $\neg\varphi\notin\Th_-(M)$, so $M\not\models_\lang\neg\varphi$. Hence, as $\varphi$ is generic, we conclude $M\models_\lang\theory_{\pm}$. \end{proof}
\end{rmk} 
\begin{rmk} Every negative formula is universal, so it is in particular $\Pi_1${\hyp}local. Hence, \cref{l:locally satisfiable} could be used to show that a negative local theory is locally satisfiable. However, in general, $\theory_{\pm}$ is not $\Pi_1${\hyp}local, so \cref{l:locally satisfiable} cannot be used to show that $\theory$ is weakly complete. 
\end{rmk}

\section{Local positive closedness} \label{s:local positive closedness} From now on, unless otherwise stated, fix a local language $\lang$ and a locally satisfiable negative local theory $\theory$. 

A homomorphism $f\map A\to B$ between two structures is a \emph{positive embedding}\footnote{We use the term ``positive embedding'' by analogy with ``elementary embedding''. In the positive logic literature, it is sometimes called ``immersion'' (e.g. \cite{yaacov2007fondements}). We prefer to avoid the term ``immersion'': normally, in other areas of mathematics, immersions are less restrictive than embeddings (every embedding is an immersion, but not the other way around), whereas the exact opposite is true here. Moreover, it is likely to clash with other uses of the term (e.g. immersions of graphs).}
 if $B\premodels \varphi(f(a))$ implies $A\premodels\varphi(a)$ for any positive formula $\varphi(x)$ and $a\in A^x$. A \emph{positive substructure}\footnote{We use the term ``positive substructure'' by analogy with ``elementary substructure''. In the positive logic literature, it is sometimes called ``immersed substructure''.} is a substructure $A\leq B$ such that the inclusion $\inc\map A\to B$ is a positive embedding. We say that a structure $A$ is \emph{positively closed}\footnote{In the positive logic literature, \emph{positive closedness} is often simply called \emph{existential closedness} (for instance, in \cite{yaacov2003positive,yaacov2007fondements,hrushovski2022definability,hrushovski2022lascar}). The reason is that, when classical first{\hyp}order logic is presented as a particular case of positive logic (via Morleyization), the classical existential closedness becomes positive closedness. However, in general, we consider this abuse of terminology confusing.} for a class of structures $\mathcal{K}$ if every homomorphism $f\map A\to B$ with $B\in\mathcal{K}$ is a positive embedding.

%\begin{rmk} The composition of positive embeddings is a positive embedding. Therefore, structures with positive embeddings form a category. \end{rmk}

%\begin{nota} We say that $B$ \emph{continues} $A$ if there is a homomorphism $h\map A\to B$.\end{nota}

\begin{defi}\label{d:locally positively closed}
Let $\Gamma$ be a set of sentences. We say that $M$ is a \emph{locally positively closed (local) model} of $\Gamma$, written $M\models^\pc_\lang\Gamma$, if $M$ is a local model of $\Gamma$ positively closed for the class $\Mod(\Gamma/\lang)$ of local models of $\Gamma$; we write $\Mod^\pc(\Gamma/\lang)$ for the class of locally positively closed models of $\Gamma$. We say that $M$ is \emph{locally positively closed} if it is a locally positively closed model of $\Th_-(M)$. We typically omit $\lang$ if it is clear from the context. \end{defi}

Given two subsets $\Gamma_1$ and $\Gamma_2$ of first{\hyp}order sentences, we write $\Gamma_1\models^\pc\Gamma_2$ if $N\models_\lang \Gamma_2$ for any $N\models^\pc_\lang\Gamma_1$. 

\begin{nota} We say that $M$ is a \emph{(non{\hyp}local) positively closed model of $\Gamma$}, written $M\premodelspc\Gamma$, if $M$ is a model of $\Gamma$ and is positively closed for the class $\Mod_\star(\Gamma)$. We say that $\Gamma_1\premodelspc\Gamma_2$ if $N\premodels\Gamma_2$ for any $N\premodelspc\Gamma_1$ 
\end{nota}
\begin{rmk}\label{r:homomorphism and positive embeddings} Recall that homomorphisms preserve satisfaction of positive formulas, i.e. if $f\map A\to B$ is a homomorphism and $A\premodels\varphi(a)$ with $\varphi(x)$ a positive formula and $a\in A^x$, then $B\premodels\varphi(f(a))$. Thus, when $f\map A\to B$ is a positive embedding, we have $A\premodels\varphi(a)$ if and only if $B\premodels\varphi(f(a))$. In particular, positive embeddings are embeddings and every positive substructure of a local structure is also a local structure.
%\begin{proof} Easy induction on the complexity:
%\begin{enumerate}[label={•}, wide]
%\item For atomic formulas there is nothing to prove.
%\item Say $\varphi=\varphi_1\wedge\varphi_2$. As $A\premodels\varphi(a)$, we have $A\premodels\varphi_1(a)$ and $A\premodels\varphi_2(a)$. By induction hypothesis, $B\premodels\varphi_1(f(a))$ and $B\premodels\varphi_2(f(a))$, so $B\premodels\varphi(f(a))$.
%\item Say $\varphi=\varphi_1\vee\varphi_2$. As $A\premodels\varphi(a)$, we have $A\premodels\varphi_1(a)$ or $A\premodels\varphi_2(a)$. By induction hypothesis, $B\premodels\varphi_1(f(a))$ or $B\premodels\varphi_2(f(a))$, so $B\premodels\varphi(f(a))$.
%\item Say $\varphi=\exists y\mathrel{} \theta(x,y)$. As $A\premodels\varphi(a)$, there is $b\in A^y$ such that $A\premodels\theta(a,b)$. By induction hypothesis, $B\premodels\theta(f(a),f(b))$, so $B\premodels\varphi(f(a))$.
%\end{itemize}
%\end{proof}
\end{rmk}

\begin{rmk}\label{r:positive full diagram}
We recall the following easy fact from positive logic. Let $M$ be an $\lang${\hyp}structure, $N_M$ an $\lang(M)${\hyp}structure and $N$ its $\lang${\hyp}reduct. Then, $\Th_-(N_M)=\Th_-(M/M)$ if and only if $f\map M\to N$ given by $a\mapsto a^{N_M}$ is a positive embedding.  
%\begin{proof} If $a\mapsto a^{B_A}$ is a positive embedding, then $\Th_-(B_A)=\Th_-(A/A)$ by \cref{r:homomorphism and positive embeddings}. On the other hand, suppose $\Th_-(B_A)=\Th_-(A/A)$. Then, for any positive formula $\varphi(x)$ and any $a\in A^x$, we have $A\models_\lang\varphi(a)$ if and only if $B\models_\lang\varphi(f(a))$. Therefore, $f$ is a positive embedding.\end{proof}
\end{rmk}

\begin{rmk} \label{r:pc is preserved by increasing the theory} Suppose $M\models^\pc_\lang\theory$ and $M\models_\lang\theory'$ with $\theory\subseteq\theory'$. Then, trivially, $M\models^\pc_\lang\theory'$. In particular, if $M\models^\pc_\lang\theory$ for some negative local theory $\theory$, then $M$ is a locally positively closed model of $\Th_-(M)$. Furthermore, for any subset of parameters $A$, $M_A$ is a locally positively closed model of $\Th_-(M/A)$.
\end{rmk}

\begin{lem}\label{l:positively closed and extensions} Let $M\models^\pc_\lang\theory$ and $A$ a subset. Let $f\map M\to N$ be a homomorphism to $N\models_\lang\theory$. Then, $N_A\models^\pc_\lang\Th_-(M/A)$ if and only if $N\models^\pc_\lang\theory$.  
\end{lem} 
\begin{proof} Since $M\models^\pc_\lang\theory$, we get that $f$ is a positive embedding. By \cref{r:positive full diagram}, we conclude that $N_A\models_\lang\Th_-(M/A)$. If $N\models^\pc_\lang\theory$, since $\theory\subseteq \Th_-(M/A)$, we conclude that $N_A\models^\pc_\lang\Th_-(M/A)$ by \cref{r:pc is preserved by increasing the theory}. Conversely, suppose $N_A\models^\pc_\lang\Th_-(M/A)$ and $g\map N\to B$ is a homomorphism to a local model $B\models_\lang\theory$. Then, $g\circ f\map M\to B$ is a homomorphism, so a positive embedding. By \cref{r:positive full diagram}, taking the respective $\lang(A)${\hyp}expansion, $B_A\models_\lang\Th_-(M/A)$ and $g\circ f$ is an $\lang(A)${\hyp}homomorphism. In particular, $g\map N_A\to B_A$ is a homomorphism. As $N_A\models^\pc_\lang\Th_-(M/A)$ we conclude that $g\map N_A\to B_A$ is a positive embedding, so $g\map N\to B$ is a positive embedding too. 
\end{proof}

\begin{lem} \label{l:local positive atomic are positively closed}  Let $M$ be an $\lang${\hyp}structure. Suppose that for every element $a$ in $M$ there is a positive formula $\theta_a$ such that $a$ is the unique element realising it. Then, $M$ is locally positively closed.
\end{lem}
\begin{proof} Let $f\map M\to N$ be a homomorphism to $N\models_\lang\Th_-(M)$. Suppose $N\models_\lang \varphi(f(a))$ for some positive formula $\varphi(x)$ and $a=(a_i)_{i<n}\in M^x$. Since homomorphisms preserve satisfaction of positive formulas, $N\models_\lang\theta_a(f(a))$ where $\theta_a(x)\coloneqq\bigwedge\theta_{a_i}(x_i)$. Then, $N\models_\lang\exists x\mathrel{}(\theta_a(x)\wedge\varphi(x))$, so $\neg\exists x\mathrel{}(\theta_a(x)\wedge\varphi(x))\notin\Th_-(M)$, concluding $M\models_\lang\exists x\mathrel{}(\theta_a(x)\wedge\varphi(x))$. Now, by assumption, $a$ is the unique element realising $\theta_a$, so $M\models_\lang\varphi(a)$.
\end{proof}

Let us recall (see \cite[Definition 1.13]{yaacov2003positive}) that a \emph{direct system of (local) structures} is a sequence $(M_i,f_{i,j})_{i,j\in I, i\succeq j}$ indexed by a direct set $(I,\prec)$ where each $M_i$ is a (local) structure, each $f_{i,j}\map M_j\to M_i$ is a homomorphism for each $i,j\in I$ with $i\succeq j$ and $f_{i,j}\circ f_{j,k}=f_{i,k}$ for any $i\succeq j\succeq k$. The \emph{direct limit structure}\footnote{Rigorously speaking, the direct limit is this structure together with the corresponding coprojections.} $M=\underrightarrow{\lim}\,M_i$ is the structure with universe $\rfrac{\bigsqcup M^s_i}{\sim}$ for each sort and interpretation $\rfrac{\bigsqcup R(M_i)}{\sim}$ for each relation symbol; where $a\sim b$ for $a\in M_i$ and $b\in M_j$ if there is some $k\succeq i,j$ such that $f_{k,i}(a)=f_{k,j}(b)$. 
%Let $N$ be a structure and $(g_i)_{i\succeq i_0}$ a sequence of homomorphisms $g_i\map N\to M_i$ with $i_0\in I$ such that $g_i=f_{i,j}\circ g_j$ for any $i\succeq j\succeq i_0$. The \emph{direct limit homomorphism} $\underrightarrow{\lim}\, g_i\map N\to M$ is the homomorphism defined by $a\mapsto [g_i(a)]_{\sim}$ for some (any) $i\succeq i_0$. In particular, 
We write $f_i$ for the canonical \emph{coprojection} of $M_i$ to $M$, which is defined by $f_i(a)=[a]_{\sim}$. 

%\begin{rmk} Recall that, indeed, direct limit structures are structures and direct limit homomorphisms are homomorphisms.\end{rmk}
We recall the following basic facts:
\begin{lem} \label{l:lemma direct limit} Let $(M_i,f_{i,j})_{i,j\in I, i\succ j}$ be a direct system of $\lang${\hyp}structures and $M$ the direct limit structure. Let $\varphi(x)$ be a formula and $a\in M^x$.
\begin{enumerate}[label={\rm{(\arabic*)}}, ref={\rm{\arabic*}}, wide]
\item \label{itm:lemma direct limit:1} If $\varphi$ is a positive formula, then $M\premodels\varphi(a)$ if and only if there is $i\in I$ such that $M_i\models_\lang\varphi(\widetilde{a})$ for $\widetilde{a}\in M^x_i$ with $a=f_i(\widetilde{a})$.
\item \label{itm:lemma direct limit:2} If $\varphi$ is an existential formula, then $M\models_\lang \varphi(a)$ implies that there is $i\in I$ with $a=f_i(\widetilde{a})$ for some $\widetilde{a}\in M^x_i$ such that $M_i\models_\lang\varphi(\widetilde{a})$.
\end{enumerate}
In particular, $M$ is a model of $\bigcap \Th_{\forall}(M_i)$. 
\end{lem}
\begin{proof} See \cite[Lemma 1.14]{yaacov2003positive}: {\rm{(\ref*{itm:lemma direct limit:1})}} and {\rm{(\ref*{itm:lemma direct limit:2})}} follow by an easy induction on the complexity. 
By {\rm{(\ref*{itm:lemma direct limit:2})}}, we have that $M$ is a model of $\bigcap \Th_{\forall}(M_i)$.
\end{proof}

\begin{lem}\label{l:direct limit} Let $(M_i,f_{i,j})_{i,j\in I, i\succ j}$ be a direct system of local models of $\theory$. Then, the direct limit structure $M$ is a local model of $\theory$. Furthermore, let $h\map M\to A$ be a homomorphism to another local $\lang${\hyp}structure. Suppose $h_i\coloneqq h\circ f_i\map M_i\to A$ is a positive embedding for every $i$, where $f_i\coloneqq \underrightarrow{\lim}_{j\succeq i}f_{j,i}$. Then, $h$ is a positive embedding. In particular, if $M_i\models^\pc_\lang\theory$ for each $i\in I$, then $M\models^\pc_\lang\theory$.
\end{lem}
\begin{proof} By \cite[Proposition 1.15]{yaacov2003positive}.%By \cref{l:lemma direct limit}, as $\theory\wedge\theory_\loc$ is a universal set of sentences satisfied by $M_i$ for each $i$, we know that $M$ satisfies $\theory$ and axioms \cref{itm:axiom 1,itm:axiom 2,itm:axiom 3,itm:axiom 4}. It remains to check axiom \cref{itm:axiom 5}. Let $a,b$ be elements in $M$ of the same sort. Suppose $a=f_i(\widetilde{a}')$ and $b=f_j(\widetilde{b}')$ for some $\widetilde{a}'$ in $M_i$ and $\widetilde{b}'$ in $M_j$. Take $k\succeq i,j$, so $a=f_k(\widetilde{a})$ and $b=f_k(\widetilde{b})$ with $\widetilde{a}=f_{k,i}(\widetilde{a}')$ and $\widetilde{b}=f_{k,j}(\widetilde{b}')$. As $M_k$ is a local structure, there is a locality relations $\dd$ such that $\dd^{M_k}(\widetilde{a},\widetilde{b})$. Hence, $\dd(a,b)$.
%
%Now, suppose that, for every $i\in I$, $h_i$ is  a positive embedding. Let $\varphi(x)\in\For^x_+(\lang)$ be a positive formula, $a\in M^x$ and assume $A\models_\lang\varphi(h(a))$. Let $i\in I$ be such that $a=f_i(\widetilde{a})$ for some $\widetilde{a}\in M^x_i$. As $h_i\map M_i\to A$ is a positive embedding and $A\models_\lang\varphi(h_i(\widetilde{a}))$, we conclude that $M_i\models_\lang\varphi(\widetilde{a})$. Then, by \cref{l:lemma direct limit}, we conclude that $M\models_\lang\varphi(a)$. As $\varphi$ and $a$ are arbitrary, we conclude that $h$ is a positive embedding. In particular, if $M_i\models^\pc_\lang\theory$ for every $i\in I$, we conclude that $M\models^\pc_\lang\theory$.  
\end{proof}
The following theorem, and its proof, is standard (see for instance \cite[Lemma 1.20]{yaacov2003positive}). For the sake of completeness, we offer it here in full detail.
\begin{theo} \label{t:positively closed} Let $N$ be a local model of $\theory$. Then, there exists a homomorphism $f\map N\to M$ to $M\models^\pc_\lang\theory$.
\end{theo}
\begin{proof} We define by recursion a particular sequence $(M_i,f_{i,i'})_{i'<i\in \omega}$ of local models of $\theory$ and homomorphisms $f_{i,i'}\map M_{i'}\to M_i$ with $M_0=N$, in such a way that the resulting direct limit structure $M$ is a locally positively closed model of $\theory$. Note that by \cref{l:direct limit}, we already know that $M$ is a local model of $\theory$ and the coprojection $f\coloneqq f_0$ is a homomorphism.

Given $M_i$ and $f_{i,i'}$, we now define $M_{i+1}$ and $f_{i+1,i'}$. Let $\{(\varphi_j,a_j)\}_{j\in\lambda_i}$ be an enumeration of all pairs where $\varphi_j$ is a positive formula and $a_j$ is a tuple in $M_i$ on the free variables of $\varphi_j$. We define by recursion a sequence $(M^j_i,g^{j,k}_i)_{k<j\in\lambda_i}$ of local models of $\theory$ and homomorphisms $g^{j,k}_i\map M^{k}_i\to M^j_i$ as follows. Set $M^0_i=M_i$. 

Assume $M^j_i$ and $g^{j,k}_i$ are already defined. We check if there is a local model $N\models_\lang \theory$ and a morphism $g\map M^j_i\to N$ such that $N\models_\lang\varphi_j(g(g^{j,0}_i(a_j)))$. If $N$ and $g$ exist, we take $M^{j+1}_i=N$ and $g^{j+1,k}_i=g\circ g^{j,k}_i$. Otherwise, we take $M^{j+1}_i=M^j_i$ and $g^{j+1,k}_i=g^{j,k}_i$.

For $\delta\in\lambda_i$ limit, we take $M^\delta_i$ and $g^{\delta,k}_i$ as direct limit of $(M^j_i,g^{j,k}_i)_{k<j\in\delta}$. In this case, we have that $M^{\delta}_i$ is a local model of $\theory$  and $g^{\delta,k}_i$ is a homomorphism by \cref{l:direct limit}.

Given $(M^j_i,g^{j,k}_i)_{k<j\in\lambda_i}$ we take $M_{i+1}$ and $g^k_i$ to be the respective direct limit. By \cref{l:direct limit}, we have that $M_{i+1}$ is a local model of $\theory$ and $g^k_i\map M^k_i\to M_{i+1}$ is a homomorphism. We take $f_{i+1,i}=g^0_i$ and $f_{i+1,i'}=f_{i+1,i}\circ f_{i,i'}$. 

Now, it remains to show that, with this particular construction, the resulting structure $M$ is a locally positively closed model of $\theory$. Let $B$ be a local model of $\theory$ and $h\map M\to B$ a homomorphism. Suppose that there is a positive formula $\varphi(x)$ and an element $a$ in $M$ such that $B\models_\lang \varphi(h(a))$. By construction, there is some $i\in\N$ and $\widetilde{a}$ in $M_i$ such that $a=f_i(\widetilde{a})$. Hence, there is $j\in \lambda_i$ such that $\widetilde{a}=a_j$ and $\varphi=\varphi_j$. As $h\map M\to B$ is a homomorphism, we have that $g=h\circ f_{i+1}\circ g^j_i\map M^j_i\to B$ is a homomorphism to a local model of $\theory$ such that $B\models_\lang\varphi_j(g(g^{j,0}_i(a_j)))$. Therefore, $M^{j+1}_i\models_\lang \varphi(g^{j+1,0}_i(\widetilde{a}))$, concluding that $M_{i+1}\models_\lang \varphi(f_{i+1,i}(\widetilde{a}))$, so $M\models_\lang \varphi(a)$. As $\varphi$, $a$, $B$ and $h$ are arbitrary, we conclude that $M$ is a locally positively closed model of $\theory$.
\end{proof}
\begin{rmk} \label{r:lowenheim skolem pc} Applying downwards L\"{o}wenheim{\hyp}Skolem in each step of the previous proof, we can take $|M^{j+1}_i|\leq |\lang|+|M^j_i|$ for each $i\in \omega$ and $j\in\lambda_i$. As $\lambda_i\leq |\lang|+|M_i|$, we can get $|M_{i+1}|\leq |M_i|$. Therefore, we can add $|M|\leq |\lang|+|N|$ to \cref{t:positively closed}.
\end{rmk}

\begin{defi} \label{d:denials} Let $\varphi$ be a positive formula. A \emph{denial} of $\varphi$ is a positive formula $\psi$ such that $\theory\models_\lang \neg \exists x\mathrel{} (\varphi\wedge\psi)$ --- note that since $\theory$ is closed under local implications and $\neg \exists x\mathrel{} (\varphi\wedge\psi)$ is negative, we get in fact that $\neg \exists x\mathrel{} (\varphi\wedge\psi)\in\theory$. We write $\theory\models_\lang\varphi\perp\psi$ or $\varphi\perp \psi$ if $\theory$ is clear from the context. \end{defi}

So far, local positive logic has shown similar behaviour to usual positive logic. The first difference appears in \cref{l:denials}. In usual positive logic, denials of positive formulas are positive formulas. In local positive logic, on the other hand, denials of local positive formulas are still positive formulas, but not necessarily local. This technical subtlety underlines many of the singularities that local positive logic has.   

\begin{lem} \label{l:denials} Let $\varphi\in\LFor^x_+(\lang)$, $M\models^\pc_\lang\theory$ and $a\in M^x$. Then, $M\not\models_\lang\varphi(a)$ if and only if $M\models_\lang\psi(a)$ for some primitive positive formula $\psi(x)$ with $\theory\models_\lang\varphi\perp\psi$. Furthermore, we may take $\psi$ local primitive positive when $x$ is pointed.
\end{lem} 
\begin{proof} Obviously, if $M\models_\lang\psi(a)$ with $\psi\perp\varphi$, then $M\not\models_\lang\varphi(a)$. On the other hand, suppose that $M\not\models_\lang\varphi(a)$. Up to elementary equivalence, we have that 
\[\varphi(x)=\exists y\in\dd(t(x))\mathrel{} \theta(x,y),\] where $\theta(x,y)$ is quantifier free positive. Consider the natural expansion $\lang(M)$ given by adding constants (where $\BoundConst$ is extended in some natural way). Let $\underline{y}=(\underline{y}_1,\ldots,\underline{y}_n)$ be new constant symbols and consider the expansion $\lang(M,\underline{y})$ of $\lang(M)$ where $\BoundConst_{\underline{y},t^M(a)}=\dd$ --- and we also define using $\ast$ the bounds $\BoundConst_{y_i,y_j}$ and $\BoundConst_{y_i,c}$ for any constant $c$ in $\lang(M)$ not in $t^M(a)$. Consider the universal set of $\lang(M,\underline{y})${\hyp}sentences 
\[\Gamma\coloneqq \theory\wedge\Diag_{\at}(M)\wedge \theta(a,\underline{y}).\]
Note that $\Gamma$ is not locally satisfiable. Indeed, in order to reach a contradiction, suppose that $N_{M,\underline{y}}\models_\lang\Gamma$. Then, we get a homomorphism $h\map M\to N$ to a local model $N\models_\lang\theory$ such that $N\models_\lang\varphi(h(a))$ witnessed by $b=\underline{y}^{N_{\underline{y}}}$. However, as $M$ is a locally positively closed model of $\theory$, we conclude that $M\models_\lang\varphi(a)$, contradicting our initial hypothesis. 

As $\Gamma$ is universal, by \cref{l:locally satisfiable}, we conclude that $\Gamma\wedge \theory_\loc(\lang(M,\underline{y}))$ is not finitely satisfiable. Take $\Lambda(a,c)\subseteq \Diag_{\at}(M)$ finite, where $c$ are some possible extra parameters, such that 
\[\theory\wedge\theory_\loc\wedge \dd(t(a),\underline{y})\wedge \Lambda(a,c)\wedge \theta(a,\underline{y})\]
is not satisfiable. If $x$ is pointed, take the subtuple $a_c$ of $a$ on the sort of $c$ (adding constants from $\lang$ if necessary) and consider the local primitive positive formula $\psi(x)\coloneqq\exists z\in\BoundConst_{a_c,c}(a_c)\mathrel{} \bigwedge\Lambda(x,z)$. Otherwise, take $\psi(x)\coloneqq\exists z\mathrel{} \bigwedge\Lambda(x,z)$. Note that $M\models_\lang \psi(a)$, as witnessed by $c$. 

Finally, we claim that $\varphi\perp\psi$. Indeed, in order to reach a contradiction, suppose otherwise, so there is $N\models_\lang\theory$ and $a'\in N^x$ such that $N\models_\lang\psi(a')$ and $N\models_\lang\varphi(a')$. Then, there are $b'\in N^y$ and $c'\in N^z$ such that $N\models\Lambda(a',c')$, $N\models_\lang\theta(a',b')$ and $N\models_\lang \dd(t(a'),b')$, which contradicts that $\theory\wedge\theory_\loc\wedge \Lambda(a,c)\wedge \dd(t(a),\underline{y})\wedge\theta(a,\underline{y})$ is not satisfiable.
\end{proof}
\begin{coro} \label{c:positive substructures} Every positive substructure of a locally positively closed  model of $\theory$ is a locally positively closed model of $\theory$.
\end{coro}
\begin{proof} Let $M\models^\pc_\lang\theory$, $A$ be a positive substructure of $M$ and $f\map A\to B$ homomorphism to $B\models_\lang\theory$. Take $\varphi\in\For^x_+(\lang)$ and $a\in A^x$ arbitrary and suppose $B\models_\lang\varphi(f(a))$. Without loss of generality, $\varphi(x)=\exists y\mathrel{} \theta(x,y)$ where $\theta$ is quantifier free positive. Hence, there is $b\in B^y$ such that $B\models_\lang\theta(f(a),b)$. Take $c\in A^y$ arbitrary. By the locality axioms, we have $\dd^{B}(b,f(c))$ for some locality relation $\dd$ on the sort of $b$. Hence, $B\models_\lang \varphi'(f(a),f(c))$ with $\varphi'(x,z)=\exists y\in \dd(z)\mathrel{} \theta(x,y)\in\LFor^{xz}_+(\lang)$. 

Now, for any $\psi\perp\varphi'$, we have that $B\not\models_\lang\psi(f(a),f(c))$. Therefore, by \cref{r:homomorphism and positive embeddings}, $A\not\models_\lang\psi(a,c)$. As $A$ is a positive substructure of $M$, that means that $M\not\models_\lang\psi(a,c)$. As $\psi$ is arbitrary, by \cref{l:denials}, we conclude that $M\models_\lang\varphi'(a,c)$. Since $A$ is a positive substructure of $M$, we conclude that $A\models_\lang\varphi'(a,c)$. In particular, $A\models_\lang\varphi(a)$. As $\varphi$ and $a$ are arbitrary, we conclude that $f$ is a positive embedding. As $f$ and $B$ are also arbitrary, we conclude that $A\models^\pc_\lang\theory$.  
\end{proof}

\begin{defi} Let $\varphi(x)$ and $\psi(x)$ be positive formulas of $\lang$. We say that $\psi$ is an \emph{approximation} to $\varphi$ if every denial of $\psi$ is a denial of $\varphi$; we write $\theory\models_\lang \varphi\leq \psi$, or $\varphi\leq \psi$ if $\theory$ is clear from the context. We say that $\psi$ and $\varphi$ are \emph{complementary} if $\varphi$ approximates every denial of $\psi$; we write $\theory\models_\lang\varphi\toprel\psi$, or $\varphi\toprel \psi$ if $\theory$ is clear from the context. 
\end{defi}

\begin{lem} \label{l:approximation and complementary} Let $\varphi(x)$ and $\psi(x)$ be positive formulas of $\lang$.
\begin{enumerate}[label={\rm{(\arabic*)}}, ref={\rm{\arabic*}}, wide]
\item \label{itm:approximation and complementary:approximation} Suppose $\psi$ is a local positive formula. Then, $\theory\models_\lang\varphi\leq\psi$ if and only if $\theory\models^\pc_\lang \forall x\mathrel{} (\varphi(x)\rightarrow\psi(x))$. Furthermore, in this case, $\theory\premodelspc \forall x\mathrel{} (\varphi(x)\rightarrow\psi(x))$.
\item \label{itm:approximation and complementary:complementary} Suppose $\varphi$ and $\psi$ are local positive formulas. Then, $\theory\models_\lang\varphi\toprel\psi$ if and only if $\theory\models^\pc_\lang\forall x\mathrel{} (\varphi(x)\vee\psi(x))$. Furthermore, in this case, $\theory\premodelspc\forall x\mathrel{} (\varphi(x)\vee\psi(x))$.
\end{enumerate}
\end{lem}
\begin{proof} \leavevmode
\begin{enumerate}[wide]
\item[{\rm{(\ref*{itm:approximation and complementary:approximation})}}] Suppose $\varphi\leq\psi$. Take $M\models^\pc_\lang\theory$ and $a\in \varphi(M)$ arbitrary. For any $\theta\perp\psi$, as $\theta\perp\varphi$, we have $M\not\models_\lang\theta(a)$. By \cref{l:denials}, we conclude then that $M\models_\lang \psi(a)$, so $a\in \psi(M)$. As $a$ is arbitrary, we conclude that $M\models_\lang\forall x\mathrel{} (\varphi\rightarrow\psi)$. As $M$ is arbitrary, we conclude $\theory\models^\pc_\lang\forall x\mathrel{} (\varphi\rightarrow\psi)$.

Suppose $\varphi\not\leq\psi$. There is then $\theta\perp \psi$ such that $\theta\not\perp\varphi$. As $\theta\not\perp\varphi$, there is $M\models_\lang\theory$ such that $M\models_\lang\exists x\mathrel{} (\theta\wedge\varphi)$. By \cref{t:positively closed}, there is a homomorphism $f\map M\to N$ to $N\models^\pc_\lang\theory$. By \cref{r:homomorphism and positive embeddings}, $N\models_\lang \exists x\mathrel{} (\theta\wedge\varphi)$. Hence, $\theta(N)\cap \varphi(N)\neq \emptyset$. As $\theta\perp\psi$, we have $\theta(N)\cap\psi(N)=\emptyset$. Therefore, we conclude that $\varphi(N)\not\subseteq \psi(N)$. In other words, $N\not\models_\lang\forall x\mathrel{} (\varphi\rightarrow\psi)$, so $\theory\not\models^\pc_\lang\forall x\mathrel{} (\varphi\rightarrow\psi)$.

Suppose now $\varphi\leq\psi$ and $M\premodelspc\theory$. Let $a\in\varphi(M)$ arbitrary, and assume towards a contradiction that $a\notin\psi(M)$. By \cite[Lemma 2.3(1)]{hrushovski2022definability} (or, for more original references, \cite[Lemma 1.26]{yaacov2003positive} and \cite[Lemma 14]{yaacov2007fondements}), there is some positive $\theta(x)$ such that $a\in\theta(M)$ and $\theory\premodels\neg\exists x\mathrel{} (\psi\wedge\theta)$. In particular, $\theory\models_\lang\neg\exists x\mathrel{} (\psi\wedge\theta)$ and thus $\theta\perp\psi$ which implies $\theta\perp\varphi$, that is $\neg\exists x\mathrel{} (\varphi\wedge\theta)\in\theory$. Since $a\in\varphi(M)\cap\theta(M)$, this contradicts $M\premodels\theory$. Therefore, we get that $a\in\psi(M)$ and so, since $M$ and $a$ are arbitrary, we conclude that $\theory\premodelspc\forall x\mathrel{} (\varphi(x)\rightarrow\psi(x))$.
\item[{\rm{(\ref*{itm:approximation and complementary:complementary})}}] Suppose $\varphi\toprel\psi$. Take $M\models^\pc_\lang\theory$ and $a\notin \psi(M)$ arbitrary. By \cref{l:denials}, we have $a\in \theta(M)$ for some $\theta\perp\psi$. As $\theta\leq\varphi$, by {\rm{(\ref*{itm:approximation and complementary:approximation})}}, we get $a\in \varphi(M)$. As $a$ is arbitrary, we conclude that $M\models_\lang\forall x\mathrel{} (\varphi\vee\psi)$. As $M$ is arbitrary, we conclude that $\theory\models^\pc_\lang\forall x\mathrel{} (\varphi\vee\psi)$.

Suppose $\varphi\not\toprel\psi$. There is then $\theta\perp\psi$ such that $\theta\not\leq\varphi$. By {\rm{(\ref*{itm:approximation and complementary:approximation})}}, there is then $M\models^\pc_\lang\theory$ such that $\theta(M)\not\subseteq \varphi(M)$. Hence, there is $a\in \theta(M)$ such that $a\notin \varphi(M)$. As $\theta\perp\psi$, we also have $a\notin \psi(M)$, so $M\not\models_\lang\forall x\mathrel{} (\varphi\vee\psi)$. In particular, $\theory\not\models^\pc_\lang\forall x\mathrel{} (\varphi\vee\psi)$.

Suppose now $\varphi\toprel\psi$ and $M\premodelspc\theory$. Let $a\notin\psi(M)$ arbitrary. Then, again by \cite[Lemma 2.3(1)]{hrushovski2022definability} (or, for more original references, \cite[Lemma 1.26]{yaacov2003positive} and \cite[Lemma 14]{yaacov2007fondements}), there is some positive $\theta(x)$ such that $a\in\theta(M)$ and $\theory\premodels\neg\exists x\mathrel{} (\psi\wedge\theta)$, which implies $\psi\perp\theta$. By assumption, this means $\theta\leq\varphi$, thus $M\premodels\forall x\mathrel{} (\theta(x)\rightarrow\varphi(x))$ by {\rm{(\ref*{itm:approximation and complementary:approximation})}}. Therefore we get $a\in\varphi(M)$ and so, since $M$ and $a$ are arbitrary, we conclude that $\theory\premodelspc\forall x\mathrel{} (\varphi(x)\vee\psi(x))$.
\qedhere
\end{enumerate}
\end{proof}

\begin{rmk} \label{r:approximation and complementary} In the case of weakly complete theories, it is easy to identify approximations and complements. If $M\models_\lang\forall x\mathrel{} (\varphi(x)\rightarrow \psi(x))$, then $\Th_-(M)\models_\lang \varphi(x)\leq \psi(x)$. Indeed, if $\Th_-(M)\models_\lang\theta\perp\psi$, then $M\models_\lang \neg\exists x\mathrel{}(\theta(x)\wedge\psi(x))$, so $M\models_\lang \neg\exists x\mathrel{} (\theta(x)\wedge\varphi(x))$, concluding that $\Th_-(M)\models_\lang \theta\perp\varphi$.

Similarly, if $M\models_\lang\forall x\mathrel{} (\varphi(x)\vee\psi(x))$, then $\Th_-(M)\models_\lang \varphi\toprel \psi$. Indeed, if $\Th_-(M)\models_\lang \theta\perp \psi$, then $M\models_\lang \neg\exists x\mathrel{} (\theta(x)\wedge\psi(x))$, so $M\models_\lang \forall x\mathrel{}(\theta(x)\rightarrow \varphi(x))$, concluding $\theory\models_\lang \theta\leq \varphi$.
\end{rmk}

\begin{defi}\label{d:full systems}
Let $\varphi\in\LFor^x_+(\lang)$. A \emph{full system of denials} of $\varphi$ in $\theory$ is a family $\Psi$ of denials of $\varphi$ such that, for any $M\models^\pc_\lang\theory$ and any $a\in M^x$, either $M\models_\lang\varphi(a)$ or $M\models_\lang\psi(a)$ for some $\psi\in \Psi$. A \emph{full system of approximations} of $\varphi$ in $\theory$ is a family $\Psi$ of approximations of $\varphi$ such that, for any $M\models^\pc_\lang\theory$ and any $a\in M^x$, if $M\models_\lang \psi(a)$ for all $\psi\in \Psi$, then $M\models_\lang\varphi(a)$.
\end{defi}
%\begin{coro} Let $x$ be a variable and $\varphi\in\LFor^x_+(\lang)$. Then, there is a full system of denials of $\varphi$. Furthermore, if $x$ is pointed, we can take it to be local positive.\end{coro}

\begin{rmk} In this paper we work with negative formulas and negative local theories. In positive logic, it is common to work with \emph{h{\hyp}inductive formulas}\footnote{They are also called ``coherent sequents'' (e.g. in \cite{johnstone2002sketches}).} too. At the level of positively closed models, h{\hyp}inductive theories are equivalent to negative theories. In the local setting, the same applies. 

A \emph{local h{\hyp}inductive formula} is a formula of the form $\forall x\mathrel{}(\varphi(x)\rightarrow \psi(x))$ where $\varphi$ is a positive formula and $\psi$ is a local positive formula; in other words, local h{\hyp}inductive formulas are precisely the ones corresponding to approximations. A \emph{(local) h{\hyp}inductive local theory} is a set of local h{\hyp}inductive sentences closed under local implication.

Given a negative local theory $\theory$, we can consider its associated h{\hyp}inductive local theory $\theory^{\mathrm{h}}\coloneqq\{\forall x\mathrel{}(\varphi(x)\rightarrow \psi(x))\sth \theory\models_\lang\varphi\leq \psi$ with $\varphi$ positive formula and $\psi$ local positive formula$\}$. Note that $\theory^{\mathrm{h}}\models\theory$ as $\forall x\mathrel{}(\varphi\rightarrow \bot)\in \theory^{\mathrm{h}}$ if and only if $\neg\varphi\in \theory$, and $\forall x\mathrel{}(\varphi\rightarrow \bot)$ is elementarily equivalent to $\neg\varphi$. By \cref{l:approximation and complementary}(\ref{itm:approximation and complementary:approximation}), we conclude that $\theory^{\mathrm{h}}$ is indeed an h{\hyp}inductive local theory. Conversely, given an h{\hyp}inductive local theory $\Gamma$, we can consider its associated negative local theory $\Gamma_-\coloneqq\{\neg\varphi\sth \forall x\mathrel{}(\varphi\rightarrow\bot)\in \Gamma\}$, which is obviously closed under local implications.

Now, we can easily check that $(\theory^{\mathrm{h}})_-=\theory$ for any negative local theory $\theory$. Dually, we can easily check that $\Gamma\subseteq (\Gamma_-)^{\mathrm{h}}$ for any h{\hyp}inductive local theory $\Gamma$ too. By \cref{l:approximation and complementary}, we know that if $M\models^\pc\theory$, then $M\models \theory^{\mathrm{h}}$, while $\theory^{\mathrm{h}}\models\theory$. Hence, $\mathcal{M}^\pc(\theory)=\mathcal{M}^\pc(\theory^{\mathrm{h}})$. \end{rmk}

\section{Completeness} \label{s:completeness} Recall that, unless otherwise stated, we have fixed a local language $\lang$ and a locally satisfiable negative local theory $\theory$.
\begin{defi}\label{d:irreducibility} \leavevmode
\begin{enumerate}[wide]
\item[\textrm{(I)}] $\theory$ is \emph{irreducible} if, for any two positive formulas $\varphi(x)$ and $\psi(y)$ with $x$ and $y$ disjoint finite variables of the same sort, we have
\[\exists x\mathrel{}\varphi(x)\in\theory_+\mathrm{\ and\ }\exists y\mathrel{}\psi(y)\in\theory_+\Rightarrow \exists x y\mathrel{} (\varphi(x)\wedge\psi(y))\in\theory_+.\]
\item[\textrm{(UI)}] Let $\dd=(\dd_s)_{s\in\Sorts}$ be a choice of a locality relation symbols for each single sort. %and write $\dd_s=\dd_{s_1}\times\cdots\times\dd_{s_n}$ for $s=(s_1,\ldots,s_n)$. 
$\theory$ is \emph{$\dd${\hyp}irreducible} if, for any two positive formulas $\varphi(x)$ and $\psi(y)$ with $x$ and $y$ disjoint finite variables of the same sort $s$, 
\[\exists x\mathrel{}\varphi(x)\in\theory_+\mathrm{\ and\ }\exists y\mathrel{}\psi(y)\in\theory_+\Rightarrow\exists x y\mathrel{} (\dd_s(x,y)\wedge\varphi(x)\wedge\psi(y))\in\theory_+.\] 
We say that $\theory$ is \emph{uniformly irreducible} if it is $\dd${\hyp}irreducible for some $\dd$.
\item[\textrm{(LJCP)}] $\theory$ has the \emph{local joint continuation property} if for any two local models $A,B\models_\lang \theory$ there are homomorphisms $f\map A\to M$ and $g\map B\to M$ to a common local model $M\models_\lang \theory$.
\item[\textrm{(wC)}] $\theory$ is \emph{weakly complete} if $\theory=\Th_-(M)$ for some $M\models^\pc_\lang\theory$.
\item[\textrm{(C)}] $\theory$ is \emph{complete} if $\theory=\Th_-(M)$ for every $M\models^\pc_\lang\theory$.% In other words, $M^\pc(\theory)\subseteq M^{\pm}(\theory)$.
\end{enumerate}
\end{defi}
\begin{rmk} Note that, by \cref{r:homomorphism and positive embeddings,t:positively closed}, our definition of weakly complete here coincides with the previously one studied. Indeed, say $\theory=\Th_-(M)$ with $M$ local structure. By \cref{t:positively closed}, there is $f\map M\to N$ with $N\models^\pc_\lang\theory=\Th_-(M)$. Now, by \cref{r:homomorphism and positive embeddings}, $\Th_-(N)\subseteq \Th_-(M)$ too, so $\Th_-(N)=\theory$.  
\end{rmk}

\begin{lem}\label{l:complete} A negative local theory $\theory$ is complete if and only if $\Th_-(M)=\Th_-(N)$ for every $M,N\models^\pc_\lang\theory$.
\end{lem}
\begin{proof} One implication is obvious. On the other hand, take $M\models^\pc_\lang\theory$ and $\varphi\in\theory_+$ arbitrary. Then, there is $A\models_\lang \theory$ witnessing $\varphi\in\theory_+$. By \cref{t:positively closed}, there is a homomorphism $f\map A\to N$ to $N\models^\pc_\lang \theory$. By \cref{r:homomorphism and positive embeddings}, $N\models_\lang \varphi$, so $N\not\models_\lang\neg\varphi$. By hypothesis, $\Th_-(N)=\Th_-(M)$, so $M\not\models_\lang\neg\varphi$, concluding $M\models_\lang \varphi$. As $\varphi$ is arbitrary, we conclude that $M\models_\lang\theory_{\pm}$, so $\theory$ is complete. 
\end{proof}

\begin{theo}\label{t:local irreducibility} The following implications are true: 
\[\mathrm{(UI)}\Rightarrow\mathrm{(LJCP)}\Rightarrow\mathrm{(C)}\Rightarrow \mathrm{(wC)}\Rightarrow\mathrm{(I)}.\]
\end{theo}
\begin{proof} \leavevmode
\begin{enumerate}[wide]
\item[\textrm{(wC)}$\Rightarrow$\textrm{(I)}] Suppose $\theory$ is weakly complete. Say $\exists x\mathrel{} \varphi(x)\in\theory_+$ and $\exists y\mathrel{}\psi(y)\in\theory_+$ with $x$ and $y$ on the same sort. Let $M$ be a local structure with $\theory=\Th_-(M)$. Then, $M\models_\lang \theory_{\pm}$, so there are $a$ and $b$ in $M$ such that $M\models_\lang \varphi(a)$ and $M\models_\lang \psi(b)$. Thus, $M\models_\lang \exists x y\mathrel{} (\varphi(x)\wedge\psi(y))$, so $\exists x y\mathrel{} (\varphi(x)\wedge\psi(y))\in\theory_+$. As $\varphi$ and $\psi$ are arbitrary, we conclude that $\theory$ is irreducible.
\item[\textrm{(C)}$\Rightarrow$\textrm{(wC)}] Trivial.
\item[\textrm{(LJCP)}$\Rightarrow$\textrm{(C)}] Let $A,B\models^\pc_\lang\theory$. By the local joint continuation property, there are homomorphisms $f\map A\to M$ and $g\map B\to M$ to a common local model $M\models_\lang\theory$. As they are locally positively closed, $f$ and $g$ are positive embeddings. By \cref{r:homomorphism and positive embeddings}, we conclude $\Th_-(A)=\Th_-(M)=\Th_-(B)$. By \cref{l:complete}, we conclude.
\item[\textrm{(UI)}$\Rightarrow$\textrm{(LJCP)}] Let $\dd=(\dd_s)_{s\in\Sorts}$ be a choice of locality relation symbols such that $\theory$ is $\dd${\hyp}irreducible. Let $A,B\models_\lang\theory$. Consider the expansion $\lang(A,B)$ of $\lang$: we extend $\BoundConst$ to pairs of elements of $A$ or pairs of elements of $B$ in the natural way and, for $a\in A^s$ and $b\in B^s$, we take $\BoundConst_{a,b}=\dd_s$. Consider 
\[\Gamma=\theory\wedge\Diag_{\at}(A)\wedge\Diag_{\at}(B).\] 

We prove that $\Gamma\wedge\theory_\loc(\lang(A,B))$ is finitely satisfiable. Take a finite subset $\Lambda\subseteq \Gamma\wedge \theory_\loc(\lang(A,B))$. Then, there are finite subsets of atomic formulas $\Theta(x)$ and $\Psi(y)$, with $x$ and $y$ of the same sort $s$, and tuples $a\in A^s$ and $b\in B^s$, with $A\models_\lang\Theta(a)$ and $B\models_\lang\Psi(b)$, such that $\Lambda\subseteq \theory\wedge\theory_\loc\wedge \Theta(a)\wedge\Psi(b)\wedge \dd_s(a,b)$. Write $\theta=\bigwedge\Theta$ and $\psi=\bigwedge \Psi$. As $A\models_\lang\theta(a)$ and $B\models_\lang\psi(b)$ with $A,B\models_\lang\theory$, it follows that $\exists x\mathrel{}\theta(x),\exists y\mathrel{}\psi(y)\in\theory_+$. By $\dd${\hyp}irreducibility, we have $\exists x y\mathrel{} (\dd_s(x,y)\wedge\theta(x)\wedge\psi(y))\in\theory_+$, so there is a local model $M\models_\lang\theory$ such that $M\models_\lang\exists x y\mathrel{} (\dd_s(x,y)\wedge\theta(x)\wedge\psi(y))$. %Take $a'\in M^s$ and $b'\in M^s$ such that $M\models_\lang\theta(a')$, $M\models_\lang\psi(b')$ and $M\models_\lang\dd_s(a',b')$. Consider the $\lang(a,b)${\hyp}expansion $M_{a,b}$ of $M$ given by $a'=\underline{a}^M$ and $b'=\underline{b}^M$. Then, $M_{a,b}\premodels\theory\wedge\theory_\loc\wedge\theta(a)\wedge\Psi(b)\wedge\dd(a,b)$. 
Thus, $\Lambda$ is satisfiable. As $\Lambda$ is arbitrary, we conclude that $\Gamma\wedge\theory_\loc(\lang(A,B))$ is finitely satisfiable. By compactness and \cref{l:locally satisfiable}, we conclude that $\Gamma$ is locally satisfiable. 

Let $M_{A,B}$ be a local model of $\Gamma$ and $M$ its $\lang${\hyp}reduct, so $M$ is a local model of $\theory$ and the maps $f\map A\to M$ and $g\map B\to M$ given by $f(a)=a^{M_{A,B}}$ and $g(b)=b^{M_{A,B}}$ are homomorphisms. Thus, $\theory$ has the local joint continuation property.\qedhere
\end{enumerate}
\end{proof}

\begin{theo} \label{t:local irreducibility in pointed languages} The following equivalences are true in pointed local languages:
\[\mathrm{(LJCP)}\Leftrightarrow\mathrm{(C)}\Leftrightarrow\mathrm{(wC)}\Leftrightarrow\mathrm{(I)}.\]
\end{theo}
\begin{proof} By \cref{t:local irreducibility}, we only need to prove the implication \textrm{(I)}$\Rightarrow$\textrm{(LJCP)}. Let $A,B\models_\lang\theory$. Consider the expansion $\lang(A,B)$ of $\lang$: we extend $\BoundConst$ to pairs of elements of $A$ or pairs of elements of $B$ in the natural way and, for $a\in A^s$ and $b\in B^s$, we take $\BoundConst_{a,b}$ given by $\ast$ from $\BoundConst_{a,c}$ and $\BoundConst_{b,c}$ with $c$ (tuple of) constants of the sort of $a$ and $b$. Note that here the pointedness assumption is crucial as, otherwise, we cannot well{\hyp}define the expansion $\lang(A,B)$. Consider 
\[\Gamma=\theory\wedge\Diag_{\at}(A)\wedge\Diag_{\at}(B).\]
We prove that $\Gamma\wedge\theory_\loc(\lang(A,B))$ is finitely satisfiable with an argument similar to the one used in \cref{t:local irreducibility} for {\rm{(UI)}$\Rightarrow$\rm{(LJCP)}}. By compactness and \cref{l:locally satisfiable}, we conclude that $\Gamma$ is locally satisfiable. 

Let $M_{A,B}$ be a local model of $\Gamma$ and $M$ its $\lang${\hyp}reduct, so $M$ is a local model of $\theory$ and the maps $f\map A\to M$ and $g\map B\to M$ given by $f(a)=a^{M_{A,B}}$ and $g(b)=b^{M_{A,B}}$ are homomorphisms. Thus, $\theory$ has the local joint continuation property.
\end{proof}
\begin{lem} \label{l:uniform irreducibility} A negative local theory $\theory$ is uniformly irreducible if and only if it is weakly complete and all its locally positively closed models are balls, i.e. for any $M\models^\pc_\lang\theory$ and sort $s$ there are $o\in M^s$ and a locality relation $\dd$ on sort $s$ such that $M^s=\dd(o)$.
\end{lem}
\begin{proof} Let $\dd=(\dd_s)_{s\in\Sorts}$ be a choice of locality relations such that $\theory$ is $\dd${\hyp}irreducible. By \cref{t:local irreducibility}, we have that $\theory$ is weakly complete. On the other hand, let $M\models^\pc_\lang\theory$. We claim that $M\models_\lang \dd_s\ast\dd_s(a,b)$ for all $a,b\in M^s$. Otherwise, by \cref{l:denials}, there is a positive formula $\psi(x,y)$ with $M\models_\lang\psi(a,b)$ and $\theory\models_\lang \psi\perp \dd_s\ast\dd_s$. Hence, $\exists x y\mathrel{} \psi(x,y)\in \theory_+$. At the same time, $\exists z w\mathrel{} z=w\in\theory_+$ for $z,w$ of the same sort as $x,y$, so $\exists x y z w\mathrel{}(\psi(x,y)\wedge\dd_s(x,z)\wedge\dd_s(y,w)\wedge z=w)\in\theory_+$ by $\dd${\hyp}irreducibility, concluding $\exists x y\mathrel{}(\psi(x,y)\wedge\dd_s\ast\dd_s(x,y))\in\theory_+$ by \cref{itm:axiom 1,itm:axiom 3}, which contradicts $\psi\perp \dd_s\ast\dd_s$. 

Conversely, suppose $\theory$ is weakly complete and every locally positively closed model of $\theory$ is a ball. Pick $M\models^\pc_\lang\theory$ with $\theory=\Th_-(M)$. Let $o=(o_s)_{s\in\Sorts}$  be a choice of elements for each sort of $M$ and $\dd=(\dd_s)_{s\in\Sorts}$ a choice of locality relations such that $M^s=\dd_s(o_s)$ for each $s\in \Sorts$. Say $\exists x\mathrel{} \varphi(x)\in\theory_+$ and $\exists y\mathrel{} \psi(y)\in \theory_+$ with $\varphi,\psi$ positive formulas and $x$ and $y$ disjoint finite variables on the same sort. Then, $M\models_\lang\varphi(a)$ and $M\models_\lang\psi(b)$ for some $a,b\in M^s$. By \cref{itm:axiom 1,itm:axiom 3}, we get $M\models_\lang \varphi(a)\wedge\psi(b)\wedge\dd_s\ast\dd_s(a,b)$, so $\exists x y\mathrel{}(\varphi(x)\wedge\psi(y)\wedge\dd_s\ast\dd_s(x,y))\in\theory_+$. In conclusion, $\theory$ is $\dd\ast\dd${\hyp}irreducible, where $\dd\ast\dd=(\dd_s\ast\dd_s)_{s\in\Sorts}$.  
\end{proof}
\begin{rmk} \label{r:uniform irreducibility uniform bound} It is worth noting that in \cref{l:uniform irreducibility} the left{\hyp}to{\hyp}right implication is actually stronger. We have shown that if $\theory$ is uniformly irreducible, then all its local positively closed models are balls of uniformly bounded size, i.e. for any sort $s$ there is a locality relation $\dd$ on $s$ such that for any $M\models^\pc_\lang\theory$ there is $o\in M^s$ with $M^s=\dd(o)$.
\end{rmk}

\begin{ex} \label{e:local irreducibility} All the implications of \cref{t:local irreducibility} are strict:
\begin{enumerate}[label={\rm{(\arabic*)}}, ref={\rm{\arabic*}}, itemsep=5pt, wide]
\item {\textrm{(wC)}$\not\Rightarrow$\textrm{(C)}} \label{itm:example local irreducibility:wC does not imply C} Consider the one{\hyp}sorted local language $\lang$ with locality predicates $\Dtt=\{\dd_n\}_{n\in \N}$ (with the usual ordered monoid structure induced by the natural numbers) and unary predicates $\{P_n,Q_n\}_{n\in\N}$. Consider the $\lang${\hyp}structure $Z$ with universe $\Z$ given by $Z\models_\lang \dd_n(x,y)\Leftrightarrow |x-y|\leq n$, $Z\models_\lang P_n(x)\Leftrightarrow x\geq n$ and $Z\models_\lang Q_n(x)\Leftrightarrow x\leq -n$ for each $n\in\N$. Let $\theory$ be its negative theory, which is obviously a local theory as $Z$ is a local structure. 

Let $I$ be the structure whose universe is a singleton $\{\infty\}$ with $I\models P_n(\infty)$ and $I\models\neg Q_n(\infty)$ for all $n\in\N$. Similarly, let $J$ be the structure whose universe is a singleton $\{-\infty\}$ with $J\models Q_n(-\infty)$ and $J\models\neg P_n(-\infty)$ for all $n\in\N$. Both are local models of $\theory$ (as they are substructures of some non{\hyp}local elementary extension of $Z$). 

By \cref{l:local positive atomic are positively closed}, $Z\models^\pc_\lang\theory$. Also, as $I$ and $J$ are singletons, it easily follows that $I,J\models^\pc_\lang\theory$. Now, $\Th_-(Z)$, $\Th_-(I)$ and $\Th_-(J)$ are obviously different, so $\theory$ is weakly complete but not complete. 

One can further check that $Z$, $I$ and $J$ are all the locally positively closed models of $\theory$ up to isomorphism. We leave it as an exercise. %Indeed, say $M\models^\pc_\lang\theory$. Then, there are three cases: either $M\models_\lang \neg \exists x\mathrel{} Q_0(x)$, or $M\models_\lang\neg\exists x\mathrel{} P_0(x)$, or $M\models_\lang \left(\exists x\mathrel{} P_0(x)\right)\wedge \left(\exists x\mathrel{} Q_0(x)\right)$. Since $Z\models_\lang \forall x\mathrel{} (P_0(x)\vee Q_0(x))$, we have $\theory\models_\lang P_0\toprel Q_0$ by \cref{r:approximation and complementary}, so the three cases are exclusive.  
\item \textrm{(LJCP)}$\not\Rightarrow$\textrm{(UI)} \label{itm:example local irreducibility:LJCP does not imply UI} Let $\lang$ and $Z$ be as in the previous example. Consider the expansion given by adding a constant for $0$ and $\theory=\Th_-(Z/0)$ the negative theory of $Z$ with a constant for $0$. By \cref{t:local irreducibility in pointed languages}, we know that $\theory$ satisfies the local joint continuation property. On the other hand, for any locality relation $\dd_n$, we have that $\theory\models_\lang \neg \exists x y\mathrel{} (\dd_n(x,y)\wedge P_n(x)\wedge Q_n(y))$. Therefore, $\theory$ is not uniformly irreducible.  
\item \textrm{(C)}$\not\Rightarrow$\textrm{(LJCP)} \label{itm:example local irreducibility:C does not imply LJCP} Consider the same one{\hyp}sorted local language $\lang$ and the $\lang${\hyp}structure $M$ with universe $2^{\omega}$ and interpretations $M\models_\lang \dd_n(\eta,\zeta)\Leftrightarrow |\{i\sth \eta(i)\neq\zeta(i)\}|\leq n$, $M\models_\lang P_n(\eta)\Leftrightarrow \eta(n)=1$ and $M\models_\lang Q_n(\eta)\Leftrightarrow \eta(n)=0$ for each $n\in\N$. Let $\theory=\Th_-(M)$ be the negative theory of $M$. Since $M\models_\lang\theory_\loc$, $\theory$ is locally satisfiable by \cref{t:local compactness}. For $\eta\in M$, write $M_\eta\coloneqq\Dtt(\eta)$.

Now, we look at $M$ as non{\hyp}local structure (i.e. as $\lang_\star${\hyp}structure). Note that by \cite[Lemma 2.24]{segel2022positive}, we get that $M$ is homomorphism universal for $\Mod_\star(\theory)$. In particular, there is a homomorphism $f\map N\to M$ for every locally positively closed model $N$ of $\theory$.

Let $N\models^\pc_\lang\theory$ be arbitrary and take $f\map N\to M$ homomorphism. Pick $a\in N$ arbitrary and set $\eta=f(a)$. Consider $M_\eta$. By \cref{l:locally satisfiable}, $M_\eta\models_\lang \theory$. Since $N$ is local, $f(N)\subseteq M_\eta$. Since $N$ is locally positively closed, we conclude that $f\map N\to M_\eta$ is a positive embedding. Take $\zeta\in M_\eta$ arbitrary and take $n=|\{m\sth \zeta(m)\neq \eta(m)\}|$, $I_0=\{m\sth \zeta(m)=0,\ \eta(m)=1\}$ and $I_1=\{m\sth \zeta(m)=1,\ \eta(m)=0\}$. Note that, then, $M_\eta\models_\lang\psi(\zeta,\eta)$ where $\psi(\zeta,\eta)=\dd_n(\zeta,\eta)\wedge\bigwedge_{m\in I_0} P_m(\zeta)\wedge\bigwedge_{m\in I_1} Q_m(\zeta)$. Hence, as $f$ is a positive embedding, there is $b\in N$ such that $N\models_\lang \psi(b,a)$. Now, by \cref{r:homomorphism and positive embeddings}, then $M_\eta\models_\lang \psi(f(b),\eta)$. However, by the pigeonhole principle, $\zeta$ is the unique element realising $\psi(x,\eta)$ in $M_\eta$, so we conclude that $f(b)=\zeta$. Since $\zeta$ is arbitrary, we conclude that $f$ is surjective, so it is an isomorphism. Since $N$ is arbitrary, we conclude that every locally positively closed model of $\theory$ is isomorphic to $M_\eta$ for some $\eta\in M$.

Let $\eta\in M$ and suppose $M\models_\lang \varphi$ with $\varphi$ a positive formula in $\lang$. For $n\in \N$, let $\lang_n$ be the reduct of $\lang$ consisting of $\{\dd_k,P_k,Q_k\}_{k\leq n}$ and take $n$ such that $\varphi$ is a formula in $\lang_n$. Consider the map $g\map M\to M_\eta$ given by $g(\xi)(k)=\xi(k)$ for $k\leq n$ and $g(\xi)(k)=\eta(k)$ for $k>n$. Obviously, $P_k(\xi)$ implies $P_k(g(\xi))$ and $Q_k(\xi)$ implies $Q_k(g(\xi))$ for any $k\leq n$. On the other hand, $|\{m\sth g(\xi)(m)\neq g(\zeta)(m)\}|\leq |\{m\sth \xi(m)\neq\zeta(m)\}|$, so $\dd_k(\xi,\zeta)$ implies $\dd_k(g(\xi),g(\zeta))$ for any $k$. Thus, $g$ is an $\lang_n${\hyp}homomorphism. By \cref{r:homomorphism and positive embeddings}, we conclude that $M_\eta\models_\lang\varphi$. As $\varphi$ is arbitrary, we conclude that $\Th_-(M_\eta)=\theory$ for any $\eta\in M$. Hence, $\theory$ is a complete local theory. 

Finally, note that $\theory$ does not have the local joint continuation property. Take $\eta,\zeta\in M$ with $\{m\sth \zeta(m)\neq \eta(m)\}$ infinite. In order to reach a contradiction, suppose that there are homomorphisms $h_1\map M_\eta\to N$ and $h_2\map  M_\zeta\to N$ with $N\models_\lang \theory$. Then, $N\models_\lang \bigwedge_{m\in I}R_m(h_1(\eta),h_2(\zeta))$ where $R_m(x,y)=(P_m(x)\wedge Q_m(y))\vee (Q_m(x)\wedge P_m(y))$ and $I=\{m\sth \zeta(m)\neq\eta(m)\}$. On the other hand, for any finite subset $J\subseteq \N$, $\theory\models_\lang \neg\exists x y\mathrel{} (\bigwedge_{m\in J}R_m(x,y)\wedge \dd_{|J|-1}(x,y))$. Thus, we get $N\not\models_\lang \dd_n(h_1(\eta),h_2(\zeta))$ for every $n\in\N$, contradicting that $N$ is a local structure.
\item \textrm{(I)}$\not\Rightarrow$\textrm{(wC)} \label{itm:example local irreducibility:I does not imply wC uncountable} Consider the one{\hyp}sorted local language $\lang$ with locality relations $\Dtt=\{\dd_n\}_{n\in\N}$ (with the usual ordered monoid structure induced by the natural numbers) and unary predicates $\{P_\alpha\}_{\alpha<\omega_1}$. Let $\mathcal{C}$ be the class of all $\lang${\hyp}structures $M$ with universe $\Z$ such that $M\models_\lang \dd_n(a,b)\Leftrightarrow |a-b|\leq n$ and such that $P_\alpha(M)\cap P_\beta(M)=\emptyset$ for all $\alpha<\beta<\omega_1$. Let $\theory \coloneqq\bigcap_{M\in\mathcal{C}}\Th_-(M)$ be the negative theory of $\mathcal{C}$.

Note that $\theory$ is a negative local theory; if $\theory\models_\lang\varphi$ for $\varphi$ a negative sentence, then, in particular, $M\models_\lang\varphi$ for any $M\in\mathcal{C}$ (since $M$ is local and $\theory\subseteq\Th_-(M)$), so $\varphi\in\theory$.

Pick $\exists x\mathrel{} \varphi(x)\in\theory_+$ and $\exists y\mathrel{} \psi(y)\in\theory_+$ arbitrary. There are $M_1,M_2\in\mathcal{C}$ such that $M_1\models_\lang\exists x\mathrel{} \varphi(x)$ and $M_2\models_\lang\exists y\mathrel{} \psi(y)$ by definition of $\theory$. Let $a\in M_1^x$ and $b\in M_2^y$ be tuples such that $M_1\models_\lang\varphi(a)$, $M_2\models_\lang\psi(b)$. Let $M_{12}$ be the $\lang${\hyp}structure with universe $\Z$ given by $M_{12}\models_\lang \dd_n(x,y)\Leftrightarrow |x-y|\leq n$ and 
\[M_{12}\models_\lang P_\alpha(c)\Leftrightarrow \begin{cases} 
M_1\models_\lang P_\alpha(c+\max(a))& \mathrm{if\ }c\leq 0,\\ 
M_2\models_\lang P_\alpha(c-\min(b)-1) & \mathrm{if\ } c>0.
\end{cases}\]
Then, the substructure $(-\infty,\max(a)]\subseteq M_1$ is isomorphic to the substructure $(-\infty,0]\subseteq M_{12}$. Similarly, the substructure $[\min(b),\infty)\subseteq M_2$ is isomorphic to the substructure $[1,\infty)\subseteq M_{12}$. Therefore, $M_{12}\models_\lang\exists x y\mathrel{} (\varphi(x)\wedge\psi(y))$. Also, the subsets $\{P_\alpha(M_{12})\}_{\alpha<\omega_1}$ are disjoint (since every negative element is contained in at most one as $M_1\in\mathcal{C}$, and likewise for the positive elements). Thus, $M_{12}\in\mathcal{C}$, concluding $\exists x y\mathrel{} (\varphi(x)\wedge\psi(y))\in\theory_+$ witnessed by $M_{12}$. As $\varphi$ and $\psi$ are arbitrary, we conclude that $\theory$ is irreducible.

On the other hand, take $N\models_\lang\theory$ arbitrary and choose a single element $a\in N$. Then, for any $n\in\N$, the set $C_n=\{\alpha<\omega_1\sth P_\alpha(N)\cap\dd_n(a)\neq\emptyset\}$ is of cardinality at most $2n+1$. Indeed, for any distinct $\alpha_0,\ldots,\alpha_{2n+1}$, we have that $\varphi=\neg\exists x y_0\ldots y_{2n+1}\mathrel{} (\bigwedge_{i\leq 2n+1} P_{\alpha_i}(y_i)\wedge\dd_n(x,y_i))$ holds in every structure in $\mathcal{C}$, so $\varphi\in\theory$. Therefore, $\bigcup_{n\in\N}C_n$ is at most countable, and so there is some $\alpha\in\omega_1\setminus\bigcup_{n\in\N}C_n$. For such $\alpha$, we get 
\[\emptyset=\bigcup_{n\in\N}(P_\alpha(N)\cap\dd_n(a))=P_\alpha(N)\cap\bigcup_{n\in\N}\dd_n(a)=P_\alpha(N).\]
In sum, $N\models_\lang \neg\exists x\mathrel{} P_\alpha(x)$, while clearly $\neg\exists x\mathrel{} P_\alpha(x)\notin\theory$. Thus, $\theory\neq\Th_-(N)$. As $N\models_\lang\theory$ is arbitrary, it follows that $\theory$ is not weakly complete.
\end{enumerate} 
\end{ex}

On the other hand, for countable languages the implication $\mathrm{(wC)}\Rightarrow\mathrm{(I)}$ of \cref{t:local irreducibility} is in fact an equivalence, as \cref{t:weakly complete iff irreducible} shows.

\begin{defi}\label{d:bound} Let $x=\{x_i\}_{i\in N}$ be a variable. Let $I$ be the subset of indexes $i\in N$ such that there is a constant symbol on the sort of $x_i$ and $J$ the set of pairs of indexes $(i,j)\in N\times N$ such that $x_i$ and $x_j$ are on the same sort. A \emph{bound} of $x$ in $\theory$ is a partial quantifier free type of the form \[\Bound(x)\coloneqq \bigwedge_{i\in I}\dd_i(x_i,c_i) \wedge\bigwedge_{(i,j)\in J} \dd_{i,j}(x_i,x_j)\] where $c_i$ is some constant on the sort of $x_i$ for $i\in I$, $\dd_i$ is a locality relation on the sort of $x_i$ for $i\in I$ and $\dd_{i,j}$ is a locality relation on the common sort of $x_i$ and $x_j$ for $(i,j)\in J$.
\end{defi}

\begin{rmk} When $x$ is a finite variable, the bounds of $x$ are quantifier free positive formulas on $x$ constructed using only locality predicates and constants.
\end{rmk}

\begin{lem}\label{l:irreducibility} Let $\varphi(x)$ be a positive formula such that $\exists x\mathrel{} \varphi(x)\in \theory_+$. Then, there is a bound $\Bound(x)$ of $x$ such that $\exists x\mathrel{} (\Bound(x)\wedge\varphi(x))\in \theory_+$.
\end{lem}
\begin{proof} Suppose $\exists x\mathrel{} \varphi(x)\in\theory_+$. Take $M$ witnessing it and $a\in M^x$ such that $M\models_\lang\varphi(a)$. As $M$ is local, there is a bound $\Bound(x)$ such that $M\models_\lang \Bound(a)$. Therefore, $M\models_\lang\exists x\mathrel{} (\Bound(x)\wedge\varphi(x))$, concluding $\exists x\mathrel{} (\Bound(x)\wedge\varphi(x))\in \theory_+$.
\end{proof}

\begin{theo}\label{t:weakly complete iff irreducible} Assume $\lang$ is countable. Then, $\theory$ is weakly complete if and only if it is irreducible. 
\end{theo}
\begin{proof} Let $I$ be the family of quantifier free primitive positive formulas $\varphi$ (up to renaming variables) such that $\exists x\mathrel{}\varphi(x)\in\theory_+$. Without loss of generality, we pick disjoint variables $x_\varphi$ for each $\varphi\in I$. Since $\lang$ is countable, fix an enumeration $\{\varphi_i(x_i)\}_{i\in\N}$ of $I$. 

By irreducibility, recursively define bounds $\Bound_i$ of $x_0\ldots x_i$ for each $i\in \N$ such that $\exists x_0\ldots x_i\mathrel{} \psi_i\in \theory_+$ where $\psi_i\coloneqq \bigwedge_{j\leq i}\varphi_j(x_j)\wedge\Bound_j(x_0\ldots x_j)$. Indeed, by \cref{l:irreducibility}, we can find $\Bound_0$. Now, given $\Bound_0,\ldots, \Bound_i$, as $\exists x_0\ldots x_i\mathrel{} \psi_i\in \theory_+$ and $\exists x_{i+1}\mathrel{}\varphi_{i+1}(x_{i+1})\in \theory_+$, by irreducibility and \cref{l:irreducibility}, there is a bound $\Bound_{i+1}(x_0\ldots x_{i+1})$ such that $\exists x_0\ldots x_{i+1}\mathrel{} \psi_{i+1}=\exists x_0\ldots x_{i+1}\mathrel{} (\psi_i\wedge \varphi_{i+1}\wedge \Bound_{i+1})\in \theory_+$.

Consider the partial type $\Sigma(x_i)_{i\in\N}=\{\psi_i(x_0,\ldots,x_i)\}_{i\in\N}$. Then, it is finitely satisfiable with $\theory\wedge\theory_\loc$ by construction, so there is some structure $M'\premodels\theory\wedge\theory_\loc$ and $a\coloneqq (a_i)_{i\in\N}$ in $M'$ realising $\Sigma$. Pick a subtuple $b$ of $a$ containing one element for each sort and consider $M=\Dtt(b)$, which is certainly a local $\lang${\hyp}structure and contains $a$. By \cref{l:locally satisfiable}, $M\models_\lang \theory$ and $M\models_\lang \Sigma(a)$. Thus, $\Th_-(M')=\theory$ as required.
\end{proof}

A \emph{(local) weak completion} of $\theory$ is a negative local theory that is minimal among the weakly complete negative local theories containing $\theory$. A \emph{(local) completion} of $\theory$ is a negative local theory that is minimal among the complete negative local theories containing $\theory$.

\begin{lem} \label{l:weak completions} Suppose $\widetilde{\theory}$ is a weak completion of $\theory$. Then, $\widetilde{\theory}=\Th_-(M)$ for some $M\models^\pc_\lang\theory$. 
\end{lem}
\begin{proof} Take a local model $N\models_\lang \widetilde{\theory}$ with $\Th_-(N)=\widetilde{\theory}$. Obviously, $N\models_\lang\theory$. By \cref{t:positively closed}, there is a homomorphism $f\map N\to M$ to $M\models^\pc_\lang\theory$. By \cref{r:locally weakly complete}, $N$ is a local model of $\widetilde{\theory}_{\pm}$. As $f$ is a homomorphism, by \cref{r:homomorphism and positive embeddings}, we have that $M\models_\lang\varphi$ for any $\varphi\in \widetilde{\theory}_+$. Hence, $\widetilde{\theory}_+\subseteq \Th_+(M)$, concluding that $\theory\subseteq \Th_-(M)\subseteq\widetilde{\theory}$. By minimality of $\widetilde{\theory}$, it follows that $\widetilde{\theory}=\Th_-(M)$.
\end{proof}

\begin{rmk} By \cref{l:weak completions}, in particular, the weak completions of $\theory$ are exactly the minimal elements of $\{\Th_-(M)\sth M\models^\pc_\lang\theory\}$. \end{rmk}

\begin{lem}\label{l:existence of weak completions} Assume $\lang$ is countable. Every locally satisfiable negative local theory has a weak completion.
\end{lem}
\begin{proof} Let $\theory_0$ be a locally satisfiable negative local theory. Consider the family of weakly complete negative local theories containing $\theory_0$. Note that it is not empty as $\theory_0$ is locally satisfiable and $\Th_-(M)$ for any $M\models_\lang\theory$ is a weakly complete extension of it. By Zorn's Lemma, it is sufficient to show that if $\mathcal{T}$ is a (non{\hyp}empty) family of weakly complete negative local theories containing $\theory_0$ linearly ordered under inclusion, then $\theory=\bigcap\mathcal{T}$ is also a weakly complete negative local theory containing $\theory_0$. Trivially, $\theory_0\subseteq \theory$. Certainly, $\theory$ is locally satisfiable, since any local model of any $\theory'\in\mathcal{T}$ is also a local model of $\theory$. It is also obviously closed under local implication. Indeed, if $\theory\models_\lang\varphi$ for some negative sentence $\varphi$, then $\theory'\models_\lang\varphi$ for any $\theory'\in\mathcal{T}$, so $\varphi\in\theory'$ for all $\theory'\in\mathcal{T}$, concluding $\varphi\in\theory$.

Finally, by \cref{t:weakly complete iff irreducible}, it is enough to show that $\theory$ is irreducible. Suppose $\exists x\mathrel{} \varphi(x)$ and $\exists y\mathrel{} \psi(y)$ are in $\theory_+$, with $\varphi(x)$ and $\psi(y)$ positive formulas and $x$ and $y$ disjoint variables on the same sort. Then, there are $\theory',\theory''\in\mathcal{T}$ such that $\exists x\mathrel{}  \varphi(x)\in \theory'_+$ and $\exists y\mathrel{}\psi(y)\in\theory''_+$. Since $\mathcal{T}$ is linearly ordered, we may assume $\theory'=\theory''$. Since $\theory'$ is irreducible by \cref{t:weakly complete iff irreducible}, we conclude $\exists x y\mathrel{}(\varphi(x)\wedge \psi(y))\in \theory'_+$. Therefore, since $\theory\subseteq\theory'$, we get $\theory'_+\subseteq\theory_+$, so $\exists x y\mathrel{}(\varphi(x)\wedge\psi(y))\in \theory_+$ as required. 
\end{proof}
\begin{rmk} Note that in \cref{l:existence of weak completions} the countability assumption on $\lang$ is only used to apply \cref{t:weakly complete iff irreducible}. Without this assumption, the previous lemma shows that every negative local theory has an \emph{irreducible completion}, i.e. a minimal irreducible negative local theory extending it. \end{rmk}

\begin{ex}\label{e:weak completions} In uncountable languages, weak completions may not exist: 

Consider the theory $\theory$ from \cref{e:local irreducibility}(\ref{itm:example local irreducibility:I does not imply wC uncountable}) and, for each $\alpha<\omega_1$, the theory $\theory_\alpha$ extending $\theory$ but only allowing $P_\beta$ for $\beta<\alpha$. Then each $\theory_\alpha$ is irreducible and, since it is effectively a theory of a countable language, it is weakly complete by \cref{t:weakly complete iff irreducible}. Every weakly complete theory $\theory'$ containing $\theory$ is by definition the theory of a model $M$ of $\theory$, thus $\{\beta<\omega_1\sth \neg\exists x\mathrel{} P_\beta(x)\notin\theory'\}$ is countable, so it has some bound $\alpha<\omega_1$. We conclude that $\theory_{\alpha+2}\subsetneq\theory'$ (since $M\models_\lang\theory_{\alpha+2}$) and, consequently, a weak completion does not exist.
\end{ex}

\begin{lem} \label{l:maximal and completions} Suppose $N\models^\pc_\lang\theory$ satisfies that $\Th_-(N)$ is complete. Then, $\Th_-(N)$ is a completion of $\theory$ and maximal in $\{\Th_-(M)\sth M\models^\pc_\lang \theory\}$. 
\end{lem}
\begin{proof} Let $\theory\subseteq \widetilde{\theory}\subseteq \Th_-(N)$ with $\widetilde{\theory}$ a complete negative local theory. Then, as $N\models_\lang \widetilde{\theory}$ and $N\models^\pc_\lang \theory$ with $\theory\subseteq \widetilde{\theory}$, it follows that $N\models^\pc_\lang\widetilde{\theory}$. Thus, $\Th_-(N)=\widetilde{\theory}$ by completeness of $\widetilde{\theory}$. 

Let $M\models^\pc_\lang\theory$ with $\Th_-(N)\subseteq \Th_-(M)$. Then, as $M\models_\lang \Th_-(N)$ and $M\models^\pc_\lang \theory$ with $\theory\subseteq \Th_-(N)$, it follows that $M\models^\pc_\lang\Th_-(N)$, so $\Th_-(N)=\Th_-(M)$ by completeness of $\Th_-(N)$.
\end{proof}

On the other hand, note that not every maximal element of $\{\Th_-(M)\sth M\models^\pc_\lang\theory\}$ is necessarily complete, as the following example shows.

\begin{ex}\label{e:maximal not complete} Let $\lang$ be the one{\hyp}sorted language with unary relations $\{S\}\cup\{P_n,Q_n\}_{n\in\N}$ and locality relations $\Dtt=\{d_n\}_{n\in\N}$ (with the usual ordered monoid structure induced by the natural numbers). Let $G=(V,E)$ be the graph with vertices $V=\{(v_1,v_2)\in\Z\times\N\sth |v_1|\leq v_2\}$ and edges $v\mathrel{E}w$ if and only if $v_2=w_2\wedge |v_1-w_1|=1$ or $|v_2-w_2|=1\wedge v_1=w_1=0$. In other words, $G$ is the tree made of paths $C_m=\{v\sth v_2=m\}$ of length $2m+1$ and one path $\Gamma=\{v\sth v_1=0\}$ connecting all the paths $C_m$ at the centre.

Construct from $G$ a structure $M$ with universe $V$ and interpretations
\[P_n(v)\Leftrightarrow v_1\geq n,\ Q_n(v)\Leftrightarrow v_2\leq -n,\  S(v)\Leftrightarrow |v_1|=v_2\]
and $M\models_\lang \dd_n(v,w)\Leftrightarrow \dd_G(v,w)\leq n$, where $\dd_G$ is the graph distance of $G$ (see \cref{image:maximal not complete}). Since $G$ is connected, $M$ is a local structure, thus $\theory=\Th_-(M)$ is a weakly complete local theory.

\begin{figure}[ht]
\begin{center}
%\begin{tikzpicture}[scale=0.9]
%\node (xL) at (-6,-3) {};  \node (yL) at (-6,3) {};  
%\node (xR) at (6,-3) {}; \node (yR) at (6,3) {};
%
%%\draw[step=0.1, gray, very thin] (xL) grid (yR);
%
%\node (p_0_0) [draw, shape=circle, fill, scale=0.15, label={below:$\scriptstyle{(0,0)}$}] at (0,-2) {};
%\node (p_0_-0) at (p_0_0) {};
%\foreach \x in {1,...,4}{
%    \foreach \y in {\x,...,-\x}{
%        \node (p_\x_\y) [draw, shape=circle, fill, scale=0.15, label={below:$\scriptstyle{(\y,\x)}$}] at (\y*1.25,\x-2) {}; 
%    };    
%    \draw (p_\x_\x) -- (p_\x_-\x) node [left, xshift=-12pt] {$C_{\x}$};
%}
%\draw (p_0_0) -- (0,2.75) node [above] {$\Gamma$};
%\draw [red, line width=4pt, draw opacity=0.5, line cap=round] (p_0_0) -- (4.2*1.25,4.2*1-2) node [pos=0.55, text opacity=1, below right] {$S$};
%\draw [red, line width=4pt, draw opacity=0.5, line cap=round] (p_0_0) -- (-4.2*1.25,4.2*1-2);
%\foreach \x in {0,...,4}{
%    \draw [fill, blue, opacity=0.1*\x+0.1] (4.2*1.25,4.2*1-2) -- (\x*1.25,4.2*1-2) node[text opacity=1, above right] {$P_{\x}$} -- (p_\x_\x) -- cycle;
%    \draw [fill, blue, opacity=0.1*\x+0.1] (-4.2*1.25,4.2*1-2) -- (-\x*1.25,4.2*1-2) node[text opacity=1, above left] {$Q_{\x}$} -- (p_\x_-\x) -- cycle;
%}
%\end{tikzpicture}
\begin{adjustbox}{width=0.75\textwidth}
\includegraphics[scale=1]{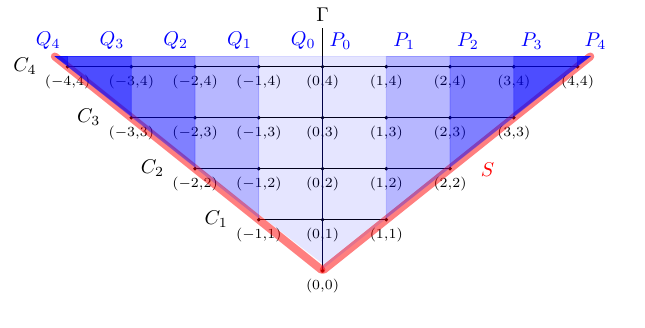}
\end{adjustbox}
\caption{} \label{image:maximal not complete}
\end{center}
\end{figure}

Let $Z$ be the structure with universe $\Z$ described in \cref{e:local irreducibility}(\ref{itm:example local irreducibility:wC does not imply C}) considered as an $\lang${\hyp}structure, i.e. $Z\models_\lang P_k(n)\Leftrightarrow n\geq k$, $Z\models_\lang Q_k(n)\Leftrightarrow n\leq -k$, $Z\models_\lang \dd_k(n,m)\Leftrightarrow |n-m|\leq k$ and $S(Z)=\emptyset$. Then, $Z\models_\lang\theory$, since every finite substructure of $Z$ is also a substructure of $M$ (contained in any long enough $C_n$). 

Let $I_S$ be the structure with a single element $\infty$ satisfying all $P_n$ as well as $S$ and no $Q_n$. Let $J_S$ be the structure with a single element $-\infty$ satisfying all $Q_n$ as well as $S$ and no $P_n$. Both are local models of $\theory$ since the types $\{S(x)\}\cup\{P_n(x)\}_{n\in\N}$ and $\{S(x)\}\cup\{Q_n(x)\}_{n\in\N}$ are finitely satisfiable in $M$.

We claim that $M$, $Z$, $I_S$ and $J_S$ are all the locally positively closed models of $\theory$. Indeed, say $N\models^\pc_\lang\theory$. Then, there are four cases: either $N\models_\lang \neg\exists x\mathrel{} Q_0(x)$, $N\models_\lang \neg \exists x\mathrel{} P_0(x)$, $N\models_\lang (\exists x\mathrel{} P_0(x))\wedge (\exists x\mathrel{} Q_0(x))\wedge(\neg\exists x\mathrel{} S(x))$ or $N\models_\lang (\exists x\mathrel{} P_0(x))\wedge (\exists x\mathrel{} Q_0(x))\wedge (\exists x\mathrel{} S(x))$. Since $M\models_\lang \forall x\mathrel{} (P_0(x)\vee Q_0(x))$, we have $\theory\models_\lang P_0\toprel Q_0$ by \cref{r:approximation and complementary}, so the four cases are exclusive.

Consider the case $N\models_\lang \neg\exists x\mathrel{} Q_0(x)$. Since $M\models_\lang \forall x\mathrel{} (Q_n(x)\rightarrow Q_0(x))$, we get $\theory\models_\lang Q_n\leq Q_0$ by \cref{r:approximation and complementary}, so $N\models_\lang\neg\exists x\mathrel{} Q_n(x)$ for all $n\in\N$ by \cref{l:approximation and complementary}(\ref{itm:approximation and complementary:approximation}). Hence, the map $h\map x\mapsto \infty$ from $N$ to $I_S$ is a homomorphism. As $N\models^\pc_\lang\theory$, we get that it is a positive embedding, so an isomorphism as it is surjective. 

Note that this implies that $I_S\models^\pc_\lang\theory$. Indeed, let $f\map I_S\to N$ be a homomorphism to $N\models^\pc_\lang\theory$, which exists by \cref{t:positively closed}. We obviously have that $\theory\models_\lang \neg\exists x y\mathrel{} (\dd_n(x,y)\wedge P_{n+1}(x)\wedge Q_0(y))$ for every $n\in\N$. Therefore, as $N\models_\lang P_n(f(\infty))$ for every $n\in\N$, we have that $N\models_\lang \neg\exists y\mathrel{} Q_0(y)$. Thus, we have the isomorphism $h\map N\to I_S$ defined above. In fact, as $|I_S|=1$, $h\circ f=\id$.

Similarly, we conclude that $N\cong J_S$ for $N\models^\pc_\lang\theory$ with $N\models_\lang \neg\exists x\mathrel{} P_0(x)$. In particular, $J_S\models^\pc_\lang\theory$. 

Consider now the case $N\models_\lang \left(\exists x\mathrel{} P_0(x)\right)\wedge \left(\exists x\mathrel{} Q_0(x)\right)\wedge \left(\neg\exists x\mathrel{} S(x)\right)$. By \cref{r:approximation and complementary}, $\theory\models_\lang P_0\toprel Q_1$ and $\theory\models_\lang P_0\perp Q_1$, so we get $N\models_\lang \forall x\mathrel{} (P_0(x)\leftrightarrow \neg Q_1(x))$ by \cref{l:approximation and complementary}(\ref{itm:approximation and complementary:complementary}). Consequently, we have the well{\hyp}defined map $h\map N\to Z$ given by
\[h(a)=\begin{cases} 
\phantom{-}\min\{k\sth N\models_\lang \exists y\mathrel{} (\dd_k(a,y)\wedge Q_0(y))\}&\mathrm{if\ }N\models_\lang P_0(a),\\ 
-\min\{k\sth N\models_\lang \exists y\mathrel{} (\dd_k(a,y)\wedge P_0(y))\} & \mathrm{if\ }N\models_\lang Q_1(a).
\end{cases}\]
We can easily check that $h$ is an isomorphism. %Indeed, this follows by essentially the same argument as in \cref{e:local irreducibility}(\ref{itm:example local irreducibility:wC does not imply C}).

Note that this implies that $Z\models^\pc_\lang\theory$. Indeed, let $f\map Z\to N$ be a homomorphism to $N\models^\pc_\lang$, which exists by \cref{t:positively closed}. Since $Z\models_\lang\left(\exists x\mathrel{} P_0(x)\right)\wedge\left(\exists x\mathrel{} Q_0(x)\right)$, we have that $N\models_\lang \left(\exists x\mathrel{} P_0(x)\right)\wedge\left(\exists x\mathrel{} Q_0(x)\right)$. It remains to show that $N\models_\lang\neg\exists x\mathrel{} S(x)$. Assume towards a contradiction that $N\models_\lang S(a)$ for some $a$ and pick $k$ such that $N\models_\lang \dd_k(a,f(0))$. 

First of all, note that $f(0)$ does not satisfy $S$, so $k\neq 0$. Indeed, $\theory\models_\lang S(x)\perp P_0(x)\wedge Q_0(x)\wedge \exists y\mathrel{} (P_1(y)\wedge\dd_1(x,y))$ as the only element in $M$ satisfying $S$, $P_0$ and $Q_0$ is $(0,0)$, whose $\dd_1${\hyp}ball contains no realisation of $P_1$ --- it only contains $(0,0)$ and $(1,0)$. Now, $Z\models_\lang P_0(0)\wedge Q_0(0)\wedge P_1(1)\wedge \dd_1(0,1)$, so $f(0)$ must not satisfy $S$. 
    
Now, for $k>0$,
\[\theory\models_\lang S(x)\wedge \exists y\mathrel{} (\dd_k(x,y)\wedge P_0(y)\wedge Q_0(y))\perp \exists y\mathrel{} (P_{2k}(y)\wedge \dd_{3k}(x,y)).\] 
Indeed, in $M$, $M\models_\lang \exists y\mathrel{} (\dd_k(v,y)\wedge P_0(y)\wedge Q_0(y))\Leftrightarrow\dd_G(v,\Gamma)=|v_1|\leq k$. Hence, we have that $M\models_\lang S(v)\wedge \exists y\mathrel{} (\dd_k(v,y)\wedge P_0(y)\wedge Q_0(y))\Leftrightarrow v_2=|v_1|\leq k$. On the other hand, if $w\in M$ satisfies $P_{2k}$, then $w_2\geq w_1\geq 2k$. Since $v_2\leq k<2k\leq w_2$, we have that $\dd_G(v,u)=|v_1|+|w_2-v_2|+|w_1|=v_2+w_2-v_2+w_1=w_1+w_2\geq 4k$, so $M\models_\lang \neg\exists x y\mathrel{} (\theta_k(x)\wedge P_{2k}(y)\wedge \dd_{3k}(x,y))$. Now, $N\models_\lang S(a)\wedge \dd_k(a,f(0))\wedge \dd_{3k}(a,f(2k))\wedge P_{2k}(f(2k))$, contradicting $N\models_\lang \theory$.

Hence, as $N\models_\lang \neg\exists x\mathrel{} S(x)$, we have the isomorphism $h\map N\to Z$ described above. In fact, as $N\models_\lang P_k(f(k))\wedge \exists y\mathrel{} (\dd_k(f(k),y)\wedge Q_0(y))$, we get that $h\circ f=\id$. 

Finally, consider the case $N\models_\lang\left(\exists x\mathrel{} P_0(x)\right)\wedge\left(\exists x\mathrel{} Q_0(x)\right)\wedge \left(\exists x\mathrel{} S(x)\right)$. Since $\theory\models_\lang P_0\toprel Q_1$ and $\theory\models_\lang P_0\perp Q_1$, we have that $N\models_\lang \forall x\mathrel{} (P_0(x)\leftrightarrow \neg Q_1(x))$ by \cref{l:approximation and complementary}(\ref{itm:approximation and complementary:approximation}). Consequently, we have the well{\hyp}defined map $h=(h_1,h_2)\map N\to M$ given by
\[h_1(a)=\begin{cases} \phantom{-}\min\{k\sth N\models_\lang \exists y\mathrel{} (\dd_k(a,y)\wedge Q_0(y))\}&\mathrm{if\ }N\models_\lang P_0(a),\\
-\min\{k\sth N\models_\lang \exists y\mathrel{} (\dd_k(a,y)\wedge P_0(y))\}&\mathrm{if\ }N\models_\lang Q_1(a);\end{cases}\]
and
\[h_2(a)= |h_1(a)|+\min\{k\sth N\models_\lang\exists y\mathrel{} (\dd_k(a,y)\wedge S(y))\}.\]
We claim that $h$ is an isomorphism.

Indeed, recall $\theory\models_\lang P_n\perp\exists y\mathrel{} (\dd_k(x,y)\wedge Q_0(y))$ for $k<n$. Thus, if $N\models_\lang P_n(a)$, then $h_1(a)\geq n$, so $M\models_\lang P_n(h(a))$. Similarly, if $N\models_\lang Q_n(a)$, then $h(a)\leq -n$, so $M\models_\lang Q_n(h(a))$.  If $N\models_\lang S(a)$, then $h_2(a)=|h_1(a)|$, so $M\models_\lang S(h(a))$.

Suppose $N\models_\lang \dd_n(a,b)$ with $h(a)=v=(v_1,v_2)$ and $h(b)=w=(w_1,w_2)$. Now, for $0\leq n\leq m$, consider the positive formulas 
\[\theta_{n,m}(x)\coloneqq P_n(x)\wedge \left(\exists y\mathrel{} (\dd_n(x,y)\wedge Q_0(y))\right)\wedge \left(\exists y\mathrel{} (\dd_{m-n}(x,y)\wedge S(y)\wedge P_m(y))\right),\] 
and
\[\theta_{-n,m}(x)\coloneqq Q_n(x)\wedge\left(\exists y\,(\dd_n(x,y)\wedge P_0(y))\right)\wedge \left(\exists y\,(\dd_{m-n}(x,y)\wedge S(y)\wedge Q_m(y))\right).\] 
Note that, in $M$, for $v\in M$, the unique element satisfying $\theta_v$ is precisely $v$. Also, for $0\leq n\leq m$, consider the formulas 
\[\psi_{n,m}(v)\coloneqq P_0(v)\wedge \left(\exists y\mathrel{} (\dd_n(v,y)\wedge Q_0(y))\right)\wedge \left(\exists y\mathrel{} (\dd_{m-n}(v,y)\wedge S(y))\right)\]
and
\[\psi_{-n,m}(v)\coloneqq Q_1(v)\wedge \left(\exists y\mathrel{} (\dd_n(v,y)\wedge P_0(y))\right)\wedge \left(\exists y\mathrel{} (\dd_{m-n}(v,y)\wedge S(y))\right)\]
Note that 
\[M\models_\lang \psi_{n,m}(v)\vee\psi_{-n,m}(v)\Leftrightarrow -n\leq v_1\leq n,\ v_2\leq m.\]
Therefore, 
\[M\models_\lang \forall x\mathrel{} ((\psi_{n,m}(x)\vee\psi_{-n,m}(x))\leftrightarrow \bigvee_{\substack{-n\leq i\leq n\\ \phantom{-}0\leq j\leq m}} \theta_{i,j}(x).\]
Thus, by \cref{r:approximation and complementary}, $\theory\models_\lang \psi_{n,m}\vee\psi_{-n,m}\leq \bigvee \theta_{i,j}$ and $\theory\models_\lang\bigvee\theta_{i,j}\leq \psi_{n,m}\vee\psi_{-n,m}$. Since $h(a)=v=(v_1,v_2)$, we get that $N\models_\lang \psi_v(a)$ and $N\not\models_\lang \psi_{n,m}(a)$ for $|n|\leq |v_1|$ or $m\leq v_2$. Hence, by \cref{l:approximation and complementary}(\ref{itm:approximation and complementary:approximation}), $N\models_\lang \theta_v(a)$. Similarly, $N\models_\lang\theta_w(b)$. Now, $M\models_\lang \neg\exists x y\mathrel{} (\theta_v(x)\wedge \theta_w(y)\wedge \dd_n(x,y))$ for $n< \dd_G(v,w)$, so $\theory\models_\lang \theta_v(x)\wedge \theta_w(y)\perp \dd_n(x,y)$ for $n< \dd_G(v,w)$, concluding that $n\geq \dd_G(h(a),h(b))$, so $M\models_\lang\dd_n(h(a),h(b))$, concluding that $h$ is a homomorphism. 

As $N$ is locally positively closed, we conclude that $h$ is a positive embedding. Now, $M\models_\lang \exists x\mathrel{} \theta_v(x)$ for all $v\in M$, so there is $a\in N$ such that $N\models_\lang \theta_v(a)$. Thus, $M\models_\lang \theta_v(h(a))$ by \cref{r:homomorphism and positive embeddings}. However, $v$ is the unique element in $M$ realising that formula, so $h(a)=v$. In conclusion, $h$ is surjective, so an isomorphism. 

Note that this implies $M\models^\pc_\lang\theory$. Indeed, let $f\map M\to N$ be a homomorphism to $N\models^\pc_\lang\theory$, which exists by \cref{t:positively closed}. Since $M\models_\lang\left(\exists x\mathrel{} P_0(x)\right)\wedge\left(\exists x\mathrel{} Q_0(x)\right)\wedge \left(\exists x\mathrel{} S(x)\right)$, we have $N\models_\lang \left(\exists x\mathrel{} P_0(x)\right)\wedge\left(\exists x\mathrel{} Q_0(x)\right)\wedge \left(\exists x\mathrel{} S(x)\right)$ by \cref{r:homomorphism and positive embeddings}. Thus, we have the isomorphism $h\map N\to M$ defined above. In fact, as $N\models_\lang \theta_v(f(v))$, we get that $h\circ f=\id$.

Hence, we have that $\Th_-(Z)$ is maximal among the negative theories of $\mathcal{M}^\pc(\theory)$. However, as noted in \cref{e:local irreducibility}(\ref{itm:example local irreducibility:wC does not imply C}), $\Th_-(Z)$ is not complete. For instance, $\Th_-(Z)$ has three different locally positively closed models, which are $Z$, $I=\{\infty\}$ and $J=\{-\infty\}$, with $I\models_\lang P_k(\infty)\wedge\neg Q_k(\infty)$ and $J\models_\lang Q_k(-\infty)\wedge\neg P_k(-\infty)$ for all $k\in\N$, and $I\models_\lang \neg S(\infty)$ and $J\models_\lang \neg S(-\infty)$. 
\end{ex}

\begin{que}\label{q:completions} Are there always completions? Is every completion of the form $\Th_-(M)$ for some $M\models^\pc_\lang\theory$? Are there always maximal elements in $\{\Th_-(M)\sth M\models^\pc_\lang\theory\}$? \end{que}

Trying to replicate the proof of \cref{l:existence of weak completions} applying some variation of Zorn's Lemma to answer \cref{q:completions} is unlikely to work. The key issue is that, as the following example shows, there are negative local theories $\theory$ and chains of complete negative local theories containing $\theory$ without a complete lower bound (containing $\theory)$.
\begin{ex} Consider \cref{e:local irreducibility}(\ref{itm:example local irreducibility:wC does not imply C}). Recall that $\lang\coloneqq\{P_n,Q_n,\dd_n\}_{n\in\N}$ and $\theory\coloneqq\Th_-(Z)$ where $Z$ is the $\lang${\hyp}structure with universe $\Z$ and interpretations $Z\models_\lang P_n(x)\Leftrightarrow x>n$, $Z\models_\lang N_n(x)\Leftrightarrow x<-n$ and $Z\models_\lang \dd_n(x,y)\Leftrightarrow |x-y|\leq n$. For each $n\in\N$, consider the substructure $C_n=[-n,n]\subseteq Z$ and the negative local theory $\theory_n\coloneqq\Th_-(C_n)$. Obviously, for each $n\in\N$, $\theory_n$ is $\dd_{2n}${\hyp}irreducible, so complete by \cref{t:local irreducibility}. Since $C_n$ embeds into $C_{n+1}$ for each $n\in\N$, we get by \cref{r:homomorphism and positive embeddings} that $\theory_{n+1}\subseteq \theory_n$. Since $Z$ is isomorphic to the direct limit of all $C_n$ (with the natural inclusions), we get that $\theory=\bigcap_{n\in\N}\theory_n$ by \cref{l:lemma direct limit}. Thus, there is no complete theory containing $\theory$ and contained in $\theory_n$ for all $n\in\N$.\end{ex}

%\begin{defi} \label{d:LJCP over parameters} Let $M$ be a local $\lang${\hyp}structure and $A$ a subset. We say that $M$ has the \emph{local joint continuation property over $A$} if $\Th_-(M/A)$ has the local joint continuation property --- alternatively, we say that $A$ \emph{satisfies the local joint continuation property in} $M$.  \end{defi} 
%By \cref{t:local irreducibility in pointed languages}, every pointed subset satisfies the local joint continuation property.  Note that being pointed is up{\hyp}preserved, i.e. if $A$ is pointed and $A\subseteq B$, then $B$ is pointed. Consequently, it is natural to wonder the following.
%\begin{que} Is the local joint continuation property up{\hyp}preserved? \end{que}

\section{Inherent locality} \label{s:inherent locality} Whenever one modifies a concept, it is useful to find the conditions under which the new concept behaves like the old one. The following definition seems to be the correct one for negative local theories and general negative theories.  Recall that, unless otherwise stated, we have fixed a local language $\lang$ and a locally satisfiable negative local theory $\theory$. 

\begin{defi}[Inherently local theories]
We say that $\theory$ is \emph{inherently local} if every $M\premodelspc\theory$ is local.
\end{defi}

To understand this definition, we note the following fact:

\begin{lem} Let $N\premodelspc \theory$. Then, $N$ satisfies \cref{itm:axiom 1,itm:axiom 2,itm:axiom 3,itm:axiom 4} of \cref{d:local structure}.
\end{lem}
\begin{proof} By \cref{l:approximation and complementary}(\ref{itm:approximation and complementary:approximation}), it is enough to show that \cref{itm:axiom 1,itm:axiom 2,itm:axiom 3,itm:axiom 4} can be specified in the form $\forall x\mathrel{} (\varphi\rightarrow\psi)$ for $\varphi,\psi$ positive.
\begin{enumerate}[wide]
\item[(\ref*{itm:axiom 1})] $\forall x y\mathrel{}(\dd(x,y)\rightarrow \dd(y,x))$ for $\dd$ an arbitrary locality relation.
\item[(\ref*{itm:axiom 2})] $\forall x y\mathrel{}(\dd_1(x,y)\rightarrow \dd_2(x,y))$ for locality relations $\dd_1\prec \dd_2$. 
\item[(\ref*{itm:axiom 3})] $\forall x y\mathrel{}((\dd_1(x,y)\wedge \dd_2(x,y))\rightarrow \dd_1\ast\dd_2(x,y))$ for $\dd_1,\dd_2$ locality relations on the same sort.
\item[(\ref*{itm:axiom 4})] $\top\rightarrow \BoundConst_{c_1,c_2}(c_1,c_2)$ for constants $c_1,c_2$ of the same sort.\qedhere
\end{enumerate}
\end{proof}

\begin{coro}\label{c:pc almost local} Let $N\premodelspc \theory$. Then, $N$ is local if and only if $N$ satisfies \cref{itm:axiom 5} of \cref{d:local structure}.
\end{coro}
\begin{theo}\label{t:inherently local}
The following are equivalent:
\begin{enumerate}[label={\rm{(\arabic*)}}, ref={\rm{\arabic*}}, wide]
\item\label{itm:inherently local:IL}  $\Mod^\pc(\theory)=\Mod^\pc_*(\theory)$.
\item\label{itm:inherently local:S_x is compact} A subset $\Gamma(x)\subseteq \For^x_+(\lang)$ is locally satisfiable in $\theory$ if and only if it is finitely locally satisfiable in $\theory$.
\item\label{itm:inherently local:S_2 is compact} A subset $\Gamma(x,y)\subseteq \For^{xy}_+(\lang)$, with $x$ and $y$ single variables of the same sort, is locally satisfiable in $\theory$ if and only if it is finitely locally satisfiable in $\theory$.
\item\label{itm:inherently local:nonlocal pc is local} $\theory$ is inherently local.
\end{enumerate}
\end{theo}
\begin{proof}\leavevmode
\begin{enumerate}[wide]
\item[(\ref*{itm:inherently local:IL})$\Rightarrow$(\ref*{itm:inherently local:S_x is compact})] Let $\Gamma(x)$ be a set of positive  $\lang$ formulas in $x$, and assume it is finitely locally satisfiable. In particular, $\Gamma$ is finitely satisfiable, so, by the usual positive compactness (compactness and \cite[Lemma 1.20]{yaacov2003positive}), there is some $M\premodelspc\theory$ and $a\in M^x$ such that $M\models_\lang\Gamma(a)$; but by assumption $M\models^\pc_\lang\theory$ and, in particular, $M\models_\lang\theory$ as required.
\item[(\ref*{itm:inherently local:S_x is compact})$\Rightarrow$(\ref*{itm:inherently local:S_2 is compact})] is trivial.
\item[(\ref*{itm:inherently local:S_2 is compact})$\Rightarrow$(\ref*{itm:inherently local:nonlocal pc is local})] By \cref{c:pc almost local}, it is enough to show that, if $a,b\in N^{s}$ for some single sort $s$, then there exists $\dd\in\Dtt^s$ such that $(a,b)\in \dd^{N}$. Assume otherwise; then by \cite[Lemma 2.3(1)]{hrushovski2022definability} (or, for more original references, \cite[Lemma 1.26]{yaacov2003positive} and \cite[Lemma 14]{yaacov2007fondements}), for any $\dd\in\Dtt^s$ there is some positive formula $\varphi_{\dd}(x,y)$ such that $(a,b)\in \varphi_{\dd}(N)$ and 
\[\theory\premodels\neg\exists x y\mathrel{} (\varphi_{\dd}(x,y)\wedge \dd(x,y)).\]
Let $\Gamma(x,y)=\left\{ \varphi_{\dd}(x,y)\right\}_{\dd\in\Dtt^s}$. Then, $\Gamma(x,y)$ is not locally satisfiable in $\theory$, since for any realisations $a',b'\in M^{s}$ with $M\models_\lang \theory$ we have $M\models_\lang\varphi_\dd(a',b')$ for any $\dd\in\Dtt^s$, so $(a',b')\notin \dd(M)$ as $\varphi_\dd\perp\dd$, contradicting \cref{itm:axiom 5}.

Thus, by assumption, $\Gamma$ is not finitely locally satisfiable in $\theory$. In other words, there exist $\dd_{1},\ldots,\dd_{n}$ such that there are no $M\models_\lang\theory$ and $a',b'\in M^{s}$ with $M\models_\lang\bigwedge_{i=1}^{n}\varphi_{\dd_{i}}(a',b')$. That means $\neg\exists x y\mathrel{} \bigwedge_{i=1}^{n}\varphi_{\dd_{i}}(x,y)\in\theory$, contradicting the assumption that $N\premodels\Gamma(a,b)\wedge\theory$. 
\item[(\ref*{itm:inherently local:nonlocal pc is local})$\Rightarrow$(\ref*{itm:inherently local:IL})] Assume $N\premodelspc\theory$. By assumption, $N$ is local; and since $N$ is positively closed for $\Mod_*(\theory)$, it is certainly positively closed for the smaller class $\Mod(\theory)$.

Conversely, assume $N\models^\pc_\lang\theory$ and $f\map N\to M$ is a homomorphism to $M\premodels\theory$. By \cite[Lemma 1.20]{yaacov2003positive}, there is a homomorphism $h\map M\to M'$ to $M'\premodelspc\theory$. By assumption, $M'$ is local, so $h\circ f$ is a positive embedding as $N\models^\pc_\lang\theory$. If $M\models_\lang \varphi(f(a))$ for a positive formula $\varphi(x)$ and $a\in N^x$, then $M'\models_\lang \varphi(hf(a))$ by \cref{r:homomorphism and positive embeddings}, concluding $M\models_\lang \varphi(a)$ as required. \qedhere
\end{enumerate}
\end{proof}

\begin{coro} \label{c:inherently local and irreducible implies the local joint continuation property}
Every inherently local irreducible negative local theory has the local joint continuation property.
\end{coro}
\begin{proof} Since $\Mod(\theory)=\Mod_\star(\theory)$, the local joint continuation property is exactly \cite[Proposition 2.9(3)]{segel2022positive} (or, for more original references, \cite[Proposition 1.44]{yaacov2003positive} and \cite[Lemma 18]{yaacov2007fondements}) for $\theory$. Thus, the claim is just {\rm{(1)}}$\Rightarrow${\rm{(3)}} in that proposition.
\end{proof}
\begin{rmk} To get an inherently local theory $\theory$ which is not irreducible, thus certainly without the local joint continuation property, just take any negative theory which is not irreducible and add a locality relation which holds between any two elements. For example, consider the theory $\{\neg\exists x\mathrel{} (P(x)\wedge Q(x))\}$ where $P,Q$ are unary relation symbols, which satisfies $\exists x\mathrel{} P(x)\in\theory_+$ and $\exists x\mathrel{} Q(x)\in\theory_+$, but clearly not $\exists x\mathrel{} (P(x)\wedge Q(x))\in\theory_+$. If one now adds $\Dtt=\{=,\top\}$ (with the trivial ordered monoid structure), then $\theory$ becomes inherently local, since all positively closed models are singletons. 
\end{rmk}

\begin{ex} There exist pointed negative local theories with the local joint continuation property which are not inherently local. For instance, consider \cref{e:local irreducibility}(\ref{itm:example local irreducibility:LJCP does not imply UI}). Then, the type $\bigwedge_{n\in \N} P_n(x)\wedge Q_n(y)$ is finitely locally satisfiable, but not locally satisfiable. Thus, $\theory$ is not inherently local by \cref{t:inherently local}, but it has the local joint continuation property by \cref{t:local irreducibility in pointed languages}.
\end{ex}

\begin{coro} \label{c:uniform irreducibility implies inherently local} Every uniformly irreducible negative local theory is inherently local.
\end{coro}
\begin{proof} Assume $\theory$ is uniformly irreducible. Using \cref{l:uniform irreducibility,r:uniform irreducibility uniform bound}, find uniform bounds $\dd=(\dd_s)_{s\in \Sorts}$ for each sort for every local positively closed model. Let $x$ and $y$ be single variables on the same sort $s$ and $\Gamma(x,y)$ a finitely locally satisfiable set of positive formulas in $\theory$.  Consider 
\[\Gamma_1(x,y)=\Gamma(x,y)\wedge\{\forall z,w \mathrel{} \dd_s(z,w)\}_{s\in \Sorts} \wedge \theory \wedge \theory_\loc\]
Consider a finite subset $\Gamma_0\subseteq \Gamma$. Since $\Gamma(x,y)$ is finitely satisfiable, there is $M\models^\pc_\lang\theory$ such that $M\models^\pc_\lang \exists x,y\mathrel{}\bigwedge\Gamma_0(x,y)$. By the choice of $\dd$ according to \cref{l:uniform irreducibility,r:uniform irreducibility uniform bound}, $M$ realizes 
\[\Gamma_0(x,y)\wedge\{\forall z,w \mathrel{} \dd_s(z,w)\}_{s\in \Sorts} \wedge \theory \wedge \theory_\loc\] 
Hence, as $\Gamma_0$ is arbitrary, $\Gamma_1(x,y)$ is finitely satisfiable. By usual compactness, there is a model of $\theory \wedge \theory_\loc$ realising $\Gamma(x,y)$ and realising $\forall z,w\mathrel{}\dd_s(z,w)$ for each sort. This is a local model realising $\Gamma(x,y)$. As $\Gamma$ is arbitrary, by \cref{t:inherently local}, $\theory$ is inherently local.
\end{proof}

\begin{ex}\label{e:IL not imply UI} The converse of \cref{c:uniform irreducibility implies inherently local} does not hold, i.e. inherently local does not imply uniform irreducibility. 

Consider the one{\hyp}sorted local language $\lang$ with locality relations $\Dtt=\{\dd_m\}_{m\in\N}$ (with the usual ordered monoid structure induced by the natural numbers) and binary predicates $\{e_n\}_{n\in\N}$. Consider the local $\lang${\hyp}structure in $\Z$ given by $\dd_m(x,y)\Leftrightarrow |x-y|\leq m$ and $e_n(x,y)\Leftrightarrow |x-y|=n$. Let $\theory$ be the negative local theory of $\Z$. Clearly, a map $f\map \Z\to \Z$ is an automorphism if and only if it is distance preserving, if and only if it is of the form $x\mapsto x+a$ or $x\mapsto -x+a$.

We claim that every type is isolated by a formula like $\bigwedge_{i<j<n}e_{n_{i,j}}(x_{i},x_{j})$; call those formulas \emph{full distance formulas}. Indeed, assume $a$ and $b$ are tuples such that $\left|a_i-a_j\right|=\left|b_i-b_j\right|$ for all $i,j<n$. Without loss of generality, $a_0<...<a_{n-1}$ (if $a_{i}=a_{j}$ for $i\neq j$, we can safely omit $a_j$). Denote $c_i\coloneqq a_{i+1}-a_i$. We have $\left|b_0-b_1\right|=\left|a_0-a_1\right|\neq 0$, so we have two options:
\begin{enumerate}[wide]
\item If $b_0<b_1$, then we have that
\[\left|b_2-b_0\right|=\left|a_2-a_0\right|=c_0+c_1=\left|b_2-b_1\right|+\left|b_1-b_0\right|.\]
Obviously, for all $x,y,z$, we have $\left|x-z\right|=\left|y-x\right|+\left|z-y\right|$ if and only if $x\leq y\leq z$ or $z\leq y\leq x$. Therefore, $b_0<b_1<b_2$ (since $b_2\leq b_1\leq b_0$ is impossible). Likewise, we get $b_0<b_1<\cdots<b_{n-1}$. Thus, 
\[b_i=b_0+\sum_{j<i}(b_{j+1}-b_j)=b_0+\sum_{j<i}(a_{j+1}-a_j)=b_0-a_0+a_i.\]
Therefore, $x\mapsto x+b_0-a_0$ maps $a$ to $b$.
\item If $b_1<b_0$, we likewise conclude $b_{n-1}<\cdots<b_0$. Thus, 
\[b_i=b_0-\sum_{j<i} (b_j-b_{j+1})=b_0-\sum_{j<i} (a_j-a_{j+1})=b_0-a_0+a_i.\]
Therefore, $x\mapsto b_0-a_0-x$ maps $a$ to $b$.
\end{enumerate}

Now, we claim that every positive formula is equivalent in $\Z$ to a disjunction of full distance formulas (i.e. a positive boolean combination of $e$ relations). 

Since $\dd_m$ is equivalent to $\bigvee_{n\leq m}e_n$, by induction on formula complexity we only need to show that, if $\theta$ is equivalent to a finite disjunction of full distance formulas, then so is $\varphi=\exists y\mathrel{}\theta$. By distributivity of $\exists$ over disjunctions, elimination of $\exists$ over dummy variables, symmetry of $e_m$ for each $m$ and the fact that $e_n(x,y)\wedge e_m(x,y)$ is inconsistent for $m\neq n$, it suffices to show that every formula of the form $\varphi(x)=\exists y\mathrel{}\bigwedge_{i<n}e_{m_{i}}(x_{i},y)$ is equivalent in $\Z$ to a finite disjunction of full distance formulas.

We prove this by induction on $n$. If $n=1$, then $\exists y\, e_{m_0}(x_0,y)$ is just equivalent to $e_0(x_0,x_0)$. Suppose it holds for $n-1>1$. Then, noting that $e_{m_i}(x_i,y)\wedge e_{m_j}(x_j,y)$ implies $e_{m_i+m_j}(x_i,x_j)\vee e_{|m_i-m_j|}(x_i,x_j)$, we get that $\varphi$ implies
\[\bigwedge_{i<j<n}(e_{m_i+m_j}(x_i,x_j)\vee e_{|m_i-m_j|}(x_i,x_j)),\] 
which is equivalent by distributivity to a finite disjunction of full distance formulas. 

Let $\Sigma(x)$ be the set of all full distance formulas in this disjunction that imply $\varphi$. Since each full distance formula isolates a type, if $\psi$ is a full distance formula that does not imply $\varphi$, then $\Z\models_\lang\neg\exists x\mathrel{} (\varphi\wedge\psi)$. If $\psi$ is a full distance formula that does not appear in the disjunction, it is inconsistent with it in $\Z$ (since different full distance formulas are inconsistent). Therefore, $\Z\models_\lang\forall x\mathrel{} (\varphi\leftrightarrow\bigvee\Sigma)$, as required.

We can deduce from this that $\theory$ is inherently local. Indeed, by \cref{t:inherently local}, we need to show that every finitely locally satisfiable set $\Gamma(x,y)$ of positive formulas on two single variables is locally satisfiable.

Assume $\Gamma$ is finitely locally satisfiable in $\theory$. Then, for any finite subset $\Gamma_0\subseteq\Gamma$ we have $\neg\exists x y\mathrel{} \bigwedge\Gamma_0\notin\theory$. That means there exist $a,b\in \Z$ realising $\Gamma_0$. If every $\varphi\in\Gamma$ is equivalent in $\Z$ to $\top$ (i.e. $e_0(x,x)$), there is nothing to show. Thus, fix some $\varphi_0(x,y)\in\Gamma$ not equivalent to $\top$. 

By the previous discussion, $\varphi_0(x,y)$ is equivalent to a disjunction $\bigvee_{i<k}\psi_{i}(x,y)$ of full distance formulas. At least one of the $\psi_i$ is realised in $\Z$ since $\varphi$ is realised. Since $e_n\wedge e_m$ for $n \neq m$ is inconsistent (and $\varphi_0$ is consistent), we may assume that every $\psi_i$ is just $e_{n_i}(x,y)$, so $\varphi_0$ is equivalent to $\bigvee_{i<k}e_{n_i}(x,y)$. Denote by $A_{\varphi_{0}}$ the set $\{ n_i\sth i<k\}$. 

Since the same holds for any $\varphi\in\Gamma$, if $\varphi_1,\ldots,\varphi_m\in\Gamma$, then $\Z\models_\lang\exists x y\mathrel{}\bigwedge_{i=1}^{m}\varphi_i$, so $\bigcap_{i=1}^{m}A_{\varphi_i}\neq\emptyset$. Since $A_{\varphi_0}$ is finite, sets of the form $A_{\varphi_0}\cap\bigcap_{i=1}^{m}A_{\varphi_i}$ have a minimum (which is not empty), thus, for any $n$ in that minimum, $e_n$ implies $\varphi$ for any $\varphi\in\Gamma$, so $\Z\models_\lang\Gamma(0,n)$, as required.

On the other hand, $\theory$ is not uniformly irreducible, since for any fixed $n_0,n_1\in\N$, we have that $\exists x_0 x_1\mathrel{} x_0=x_1$ and $\exists y_0 y_1\mathrel{} e_{n_0+n_1+1}(y_0,y_1)$ are both true in $\Z$, but 
\[\neg\exists x_0 x_1 y_0 y_1\mathrel{} (x_0=x_1\wedge e_{n_0+n_1+1}(y_0,y_1)\wedge \dd_{n_0}(x_0,y_0)\wedge \dd_{n_1}(x_1,y_1))\in\theory.\]
\end{ex}

In the context of one{\hyp}sorted languages, the reasoning in \cref{e:IL not imply UI} is somehow optimal, in the sense that it can be generalised to get a sufficient condition under which the local joint continuation property implies inherent locality.

\begin{lem} Assume $\theory$ is a negative local theory with the local joint continuation property in a one{\hyp}sorted language $\lang$. Assume further that for any atomic formula $\varphi(x,y,z)$ with $x$ and $y$ single variables, there is some $\dd\in\Dtt$ such that 
\[(\exists z\mathrel{} \varphi(x,y,z))\leq\dd(x,y).\]
Then, $\theory$ is inherently local.
\end{lem}
\begin{proof} By \cref{t:inherently local}, it suffices to show that for any $N\premodelspc\theory$ and any $a_0,a_1\in N$ there is some $\dd\in\Dtt$ such that $N\models_\lang\dd(a_0,a_1)$.  In other words, it suffices to show that $\Dtt(a)=N$ for any $a\in N$, where $\Dtt(a)$ denotes the local component at $a$. By \cref{l:approximation and complementary}, if $(\exists z\mathrel{} \varphi(x,y,z))\leq\dd(x,y)$, then $N\premodels \forall x y\mathrel{} ((\exists z\mathrel{} \varphi(x,y,z))\rightarrow\dd(x,y))$.

Let $\kappa>|N|^+$, and take some \emph{local homomorphism $\kappa${\hyp}universal model} $M$ of $\theory$, i.e. a local model $M$ of $\theory$ such that every local model of $\theory$ of size at most $\kappa$ admits a homomorphism to $M$ --- such a model exists by the local joint continuation property and \cref{l:direct limit,t:positively closed}.

Assume towards a contradiction that, for some $a_0\in N$, we have $\Dtt(a_0)\neq N$, and pick $a_1\in N\setminus\Dtt(a_0)$.  For any $\dd\in\Dtt$ there exists, by \cite[Lemma 2.3(1)]{hrushovski2022definability} (or, for more original references, \cite[Lemma 1.26]{yaacov2003positive} and \cite[Lemma 14]{yaacov2007fondements}), some quantifier free primitive positive formula $\theta_\dd(x)$ (with variable $x=\{x_i\}_{i<n}$ of length $n\geq 2$) such that 
\[\theory\premodels\neg\exists x_0 x_1\mathrel{} ((\exists x_2\ldots x_{n-1}\mathrel{} \theta_\dd(x))\wedge\dd(x_0,x_1))\] 
and 
\[N\premodels\exists x_2 \ldots x_{n-1}\mathrel{} \theta_\dd(a_0,a_1,x_2,\ldots,x_{n-1}).\] 
Fix $a_2,\ldots,a_{n-1}\in N$ such that $N\premodels\theta_\dd(a)$

Define the equivalence relation $\sim$ on $x$ by $x_i\sim x_j\Leftrightarrow a_i\in\Dtt(a_j)$. Let $F$ be the family of equivalence classes of $\sim$. By assumption, if $x_i,x_j$ appear in the same atomic subformula of $\theta_\dd$, then $x_i\sim x_j$. Therefore, we can rewrite $\theta_\dd$ as $\bigwedge_{y\in F}\theta^y_\dd(y)$ where $\theta^y_\dd$ is the conjunction of the atomic subformulas of $\theta_\dd$ with variables in $y$. Note that $\exists x\mathrel{} \theta_\dd$ is logically equivalent to $\bigwedge_{y\in F}\exists y\mathrel{} \theta_\dd^y(y)$ since equivalence classes are disjoint. For each $y\in F$, denote $N_y\coloneqq\Dtt(a_i)$ for one (any) $x_i\in y$. Note that $N_y\models_\lang\theory$ and $N_y\models_\lang\theta^y_\dd(a_y)$, where $a_y=(a_i)_{x_i\in y}$. 

For any $a$ in $N$, we have $\Dtt(a)\models_\lang\theory$ and $|\Dtt(a)|\leq|N|<\kappa$. Then, we can find a homomorphism $h_{\Dtt(a)}\map \Dtt(a)\to M$. By \cref{r:homomorphism and positive embeddings}, for any $y\in F$, $M\models_\lang\theta_\dd^y(h_{N_y}(a_y))$. Therefore, 
\[M\models_\lang\exists x_2 \ldots x_{n-1}\mathrel{} \theta_\dd(h_{\Dtt(a_0)}(a_0),h_{\Dtt(a_1)}(a_1), x_2,\ldots x_{n-1})\] and so, by choice of $\theta_\dd$, we have that $(h_{\Dtt(a_0)}(a_0),h_{\Dtt(a_1)}(a_1))\notin\dd^M$. As $\dd$ is arbitrary, we get that $h_{\Dtt(a_0)}(a_0),h_{\Dtt(a_1)}(a_1)$ are not in any locality relation with each other, contradicting that $M$ is a local structure.
\end{proof}

\bibliographystyle{alphaurl}
\bibliography{Completeness_in_local_positive_logic}

\begin{thebibliography}{BYBHU08}

\bibitem[BY03]{yaacov2003positive}
Itay Ben-Yaacov.
\newblock Positive model theory and compact abstract theories.
\newblock {\em J. Math. Log.}, 3(1):85--118, 2003.
\newblock \href {https://doi.org/10.1142/S0219061303000212}
  {\path{doi:10.1142/S0219061303000212}}.

\bibitem[BYBHU08]{yaacov2008model}
Itay Ben~Yaacov, Alexander Berenstein, C.~Ward Henson, and Alexander Usvyatsov.
\newblock Model theory for metric structures.
\newblock In {\em Model theory with applications to algebra and analysis. Vol.
  2}, volume 350 of {\em London Math. Soc. Lecture Note Ser.}, pages 315--427.
  Cambridge Univ. Press, Cambridge, 2008.
\newblock \href {https://doi.org/10.1017/CBO9780511735219.011}
  {\path{doi:10.1017/CBO9780511735219.011}}.

\bibitem[BYP07]{yaacov2007fondements}
Ita\"{\i} Ben~Yaacov and Bruno Poizat.
\newblock Fondements de la logique positive.
\newblock {\em J. Symbolic Logic}, 72(4):1141--1162, 2007.
\newblock \href {https://doi.org/10.2178/jsl/1203350777}
  {\path{doi:10.2178/jsl/1203350777}}.

\bibitem[Hru22a]{hrushovski2022lascar}
Ehud Hrushovski.
\newblock Beyond the lascar group, 2022.
\newblock \href {http://arxiv.org/abs/2011.12009v3}
  {\path{arXiv:2011.12009v3}}.

\bibitem[Hru22b]{hrushovski2022definability}
Ehud Hrushovski.
\newblock Definability patterns and their symmetries, 2022.
\newblock \href {http://arxiv.org/abs/1911.01129v4}
  {\path{arXiv:1911.01129v4}}.

\bibitem[Joh02]{johnstone2002sketches}
Peter~T. Johnstone.
\newblock {\em Sketches of an elephant: a topos theory compendium. Vol. II},
  volume~44 of {\em Oxford Logic Guides}.
\newblock The Clarendon Press, Oxford University Press, Oxford, 2002.

\bibitem[PY18]{poizat2008positive}
Bruno Poizat and Aibat Yeshkeyev.
\newblock Positive jonsson theories.
\newblock {\em Log. Univers.}, 12(1-2):101--127, 2018.
\newblock \href {https://doi.org/10.1007/s11787-018-0185-8}
  {\path{doi:10.1007/s11787-018-0185-8}}.

\bibitem[Seg22]{segel2022positive}
Ori Segel.
\newblock Positive definability patterns, 2022.
\newblock \href {http://arxiv.org/abs/2207.12449v2}
  {\path{arXiv:2207.12449v2}}.

\end{thebibliography}
\end{document}